\documentclass{amsart}
\usepackage[numbers]{natbib}
 \usepackage{epsfig}
 \usepackage{tikz}

\usepackage{graphicx,fancybox,pgf,amsmath,float,amssymb}
\usepackage{fancybox,graphicx,dsfont,pgf}
\usepackage{subfigure,color,wrapfig}

\usepackage{latexsym}

\usepackage{dsfont}
\newcommand{\RR}{\mathds{R}}

\newcommand{\NN}{\mathds{N}}

\newcommand{\eps}{\varepsilon}
\newcommand{\la}{\lambda}

\newcommand{\calA}{{\mathcal A}}
\newcommand{\calB}{{\mathcal B}}

\numberwithin{equation}{section}

\newcommand{\mE}{\mathbf{E}}
\newcommand{\mD}{\mathbf{D}}
\newcommand{\mI}{\mathbf{I}}
\newcommand{\mx}{\mathbf{x}}
\newcommand{\my}{\mathbf{y}}

\newcommand{\EE}{\mathds{E}}

\def\<{\langle}
\def\>{\rangle}
\newcommand{\LL}{\mbox{\fontfamily{phv}\selectfont L}}
\newcommand{\I}{\mbox{\fontfamily{phv}\selectfont I}}

\usepackage{color,soul}
\newcommand{\comment}[1]{{\color{blue}\shadowbox{#1}}}

\renewcommand{\comment}[1]{}

 \usepackage{framed}
\colorlet{shadecolor}{gray!30}
\newenvironment{ana}
  {\begin{leftbar}
  \begin{shaded} }
{  \end{shaded}\end{leftbar}}

 \newcommand{\arxiv}[1]{\begin{ana}
  #1\end{ana}}

\newtheorem{theorem}{Theorem}
 \newtheorem{lemma}[theorem]{Lemma}
\newtheorem{proposition}[theorem]{Proposition}

\theoremstyle{definition}

\theoremstyle{remark}
\newtheorem{remark}[theorem]{Remark}
\theoremstyle{remark}

\theoremstyle{definition}

\title[ASEP and Quadratic Harnesses]{Asymmetric Simple Exclusion Process  with open boundaries and Quadratic Harnesses}

\author{
W{\l}odek  Bryc
}
\address{
Department of Mathematics,
University of Cincinnati,
PO Box 210025,
Cincinnati, OH 45221--0025, USA}
\email{Wlodzimierz.Bryc@UC.edu}

\author{Jacek Weso{\l}owski}
\address{ Faculty of Mathematics and Information Science
Warsaw University of Technology, pl. Politechniki 1 00-661
Warszawa, Poland}
\email{wesolo@alpha.mini.pw.edu.pl}
\begin{document}


\begin{abstract}

We show that the joint probability generating function of the stationary measure of  a finite state asymmetric exclusion process
 with open boundaries can be expressed in terms of joint moments of
 Markov processes called quadratic harnesses. We use our representation to prove the large deviations principle  for the total number of particles in the system.
  We  use the generator of the Markov process
  to show how  explicit formulas for the average occupancy of a site arise for special choices of parameters.
  We also give similar representations for limits of stationary measures as the number of sites tends to infinity.

\end{abstract}
\maketitle
\arxiv{This  is an expanded version with additional details that are omitted from the published version.}
\section{Introduction}
\subsection{Asymmetric simple exclusion process with open boundaries}\label{Sec:ASEP}

Asymmetric simple exclusion process (ASEP) is a Markov model for  random particles that cannot occupy the same position, and
tend to move to the adjacent site with the rate that is larger to the right   than to the left. There are several versions of
the model; for example, \citet{spitzer1970interaction} considers ASEP on the infinite number and on the finite number of sites.
In this paper we are mostly interested in a particle system on a finite lattice of $N\geq 2$ points $S=\{1,\dots,N\}$,
 where each site $j\in S$ can have only $\tau_j=0$ or $\tau_j=1$ particles. The term ``open boundary''
refers to the fact that particles can be inserted and removed from both boundary points. This version of ASEP
appeared in \citet[Section 3]{Liggett-1975} and was extensively studied in physics,
  see e.g. \cite{bertini2002macroscopic,blythe2000exact,derrida1992exact,derrida2004current,derrida1993exact,derrida2001free,derrida2002exact}.
  A survey paper of \citet{blythe2007nonequilibrium} gives additional references and motivation.

  An informal description of the exclusion process is that particles are placed at rate $\alpha>0$
in position $1$, and at rate $\delta\geq 0$ in position $N$, provided the location is empty. Particles are also removed at rate $\gamma\geq 0$  from
location $1$ and removed  at rate $\beta>0$ from  site $N$.
Particles attempt to move to the nearest neighbor: to the right with rate 1 and to the left with rate $q\geq 0$; however they cannot move if the neighbor site is already occupied.
 Parameters $\alpha,\beta,\gamma,\delta$ describe behavior at the boundaries; parameter $q$ determines the degree of asymmetry, see Fig. \ref{Fig1}.
 Exclusion process is symmetric, when
 $q=1$; the case $q>1$ is similar to  $q<1$ due to ``particles-holes duality'', see e.g. discussion of this  point in
 \cite{uchiyama2004asymmetric}.  Throughout this paper we will assume that $q<1$.

\begin{figure}[H]
  \begin{tikzpicture}[scale=.9]
\draw [fill=black, ultra thick] (.5,1) circle [radius=0.2];
\draw [fill=black, ultra thick] (2.5,1) circle [radius=0.2];
  \draw [ultra thick] (1.5,1) circle [radius=0.2];
  \draw [ultra thick] (5,1) circle [radius=0.2];
   \draw [fill=black, ultra thick] (6,1) circle [radius=0.2];

    \draw [ultra thick] (7,1) circle [radius=0.2];
      \draw [fill=black, ultra thick] (9.5,1) circle [radius=0.2];
   \draw [ultra thick] (10.5,1) circle [radius=0.2];
     \draw[->] (-1,2.3) to [out=-20,in=135] (.5,1.5);
   \node [above right] at (-.2,2) {$\alpha$};
     \draw[->] (10.5,1.5) to [out=45,in=200] (12,2.3);
     \node [above left] at (11.2,2) {$\beta$};
            \node  at (8.25,1) {$\dots$};  \node  at (3.75,1) {$\dots$};
      \node [above] at (6.5,1.8) {$1$};
      \draw[->,thick] (6.1,1.5) to [out=45,in=135] (7,1.5);
        \node [above] at (5.5,1.8) {$q$};
            \draw[<-] (5,1.5) to [out=45,in=135] (5.9,1.5);
                 \node [above] at (10,1.8) {$1$};
                \draw[->,thick] (9.6,1.5) to [out=45,in=135] (10.4,1.5);
               \node [above] at (9,1.8) {$q$};
                 \draw[<-] (8.4,1.5) to [out=45,in=135] (9.4,1.5);
       \draw[<-] (-1,-.3) to [out=0,in=-135] (.5,0.6);
   \node [below right] at (-.2,0) {$\gamma$};
        \draw[<-] (10.5,.6) to [out=-45,in=180] (12,-.3);
   \node [below left] at (11.2,0) {$\delta$};
    \node [above] at (0.5,0) {$1$};
    \node [above] at (1.5,0) {$2$};
   \node [above] at (2.5,0) {$3$};
     \node [above] at (3.75,0) {$\dots$};
   \node [above] at (5,0) {$k-1$};
     \node [above] at (6,0) {$k$};
       \node [above] at (7,0) {$k+1$};
        \node [above] at (8.25,0) {$\dots$};
         \node [above] at (10.5,0) {$N$};
          \node [above] at (9.5,0) {$N-1$};
\end{tikzpicture}
\caption{Asymmetric Simple Exclusion Process with open boundary and  parameters $\alpha,\beta,\gamma,\delta, q$.
Black disks represent occupied sites. \label{Fig1}}
\end{figure}
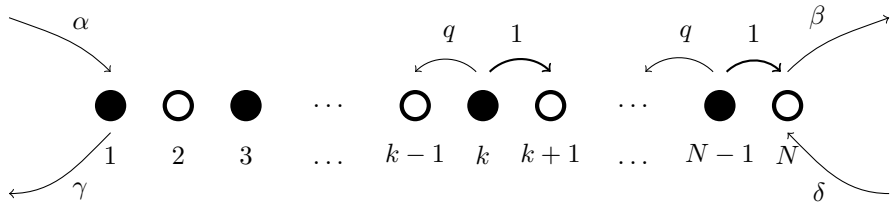

Time evolution of such a system is described by a continuous time Markov process $(\tau_1(t),\dots,\tau_N(t))$ with finite state space  $\{0,1\}^N$, see e.g. \cite[Formulas (2.1)-(2.3)]{sandow1994partially}.
We do not provide the details,
 as we are interested only in the (unique) invariant probability measure $\pi_N$ of the Markov chain $(\tau_1(t),\dots,\tau_N(t))$.
 Using the notation borrowed from statistical physics, by $\langle \cdot\rangle_N$
  we denote the average with respect to this  probability measure and  by $(\tau_1,\dots,\tau_N)\in\{0,1\}^N$ we denote
  the random vector with the invariant distribution $\pi_N$. For example, $\<\tau_k\>_N$
  is the average occupancy of site $k$ with respect to $\pi_N$.

\subsection{Overview of the main results}
\comment{Referee asked for an overview of the paper!}

Our  first main result, Theorem \ref{T-rep} in Section \ref{sect:AMP},     relates
 joint probability generating function of $\tau_1,\dots,\tau_N$  to the joint moments of a certain
Markov processes $(Z_t)_{t\geq 0}$, see \eqref{QHsolution}.
The Markov process  comes from \cite{Bryc-Wesolowski-08} where it is  called a quadratic harness. It
is described in more detail in Section \ref{AMP}. In Section \ref{Sec:MS}, we review the matrix method of \cite{derrida1993exact}, which serves as an intermediate step in the proof of Theorem \ref{T-rep}.

In Section \ref{Sect:LD}, we use  representation \eqref{QHsolution} to prove the large deviations principle  for the total number of particles in the system. Theorem \ref{T-LDP} extends the results of
\cite[Section 3.5]{derrida2003exact} to the cases where particles can leave left boundary with rate $\gamma\ne 0$ and can be inserted at the right boundary with rate $\delta\ne 0$.
The   large deviations rate function depends on the parameters of the ASEP only through the support of random variable $Z_t$, see \eqref{LL-supp}.

In Section \ref{Sect:miracle} we  give an integral formula for the average density profile when $\gamma=\delta=0$ and $\alpha+\beta>1-q$ and point out that the formula simplifies to a product
form \eqref{Schutz-like}
 when $q=0$.  The proof of  Theorem \ref{T-miracle} uses
 the generator of the Markov process $(X_t)$, which is given in Theorem \ref{Thm-gen-biPoisson}.

In Section \ref{Sect:semininf} we pass to the limit   as the number of sites $N$ tends to infinity.
The main result of this section, Theorem \ref{T-semiinfinite}, represents the joint generating function of the limiting measure of the ASEP on the semi-infinite lattice in
 terms of the joint moments of a quadratic harness.

\subsection{Matrix solution}\label{Sec:MS}
The celebrated matrix method of \citet{derrida1993exact} introduces
a probability measure on $\{0,1\}^N$ using a pair  $\mD$ and $\mE$ of infinite matrices with $\mD$ representing the occupied site,
 $\mE$ representing an empty site, and a pair of vectors, where
  $\langle W|$ is a row vector and $|V\rangle$ is a column vector. It  turns out that this is the invariant measure
  $\pi_N$ introduced in Section \ref{Sec:ASEP}.
  For our purposes it is convenient to recast the  expression from \cite{derrida1993exact}
   into the formula for the joint generating function
\begin{equation}
\label{MatrixSolution}
\left\langle \prod_{j=1}^N t_j^{\tau_j}\right\rangle_N=\frac{1}{K_N}\langle W| (\mE+t_1\mD)(\mE+t_2\mD)\dots(\mE+t_N\mD)|V\rangle.
\end{equation}
Here
 \begin{equation}
   \label{KN}
   K_N=\langle W| (\mE+ \mD)^N|V\rangle
 \end{equation} is the normalizing constant, which in the physics literature is
 called the partition function and  is usually denoted by letter $Z$;
 in this paper, as in \cite{Bryc-Wesolowski-08},  we reserve letter $Z$ for a Markov process that we  introduce in Section \ref{AMP}.

In ref. \cite{derrida1993exact} the authors verify  that $\langle \prod_{j=1}^N t_j^{\tau_j}\rangle_N$ is indeed the
 probability generating function of the invariant
 distribution $\pi_N$ of the exclusion process with parameters $\alpha,\beta,\gamma,\delta, q$, provided that
   the matrices and the vectors satisfy relations
 \begin{eqnarray}
  \mD\mE-q\mE\mD&=&\mD+\mE ,\label{q-comm-Derrida}\\
\langle W|(\alpha \mE-\gamma \mD)&=&\langle W| ,\label{W}\\
(\beta \mD-\delta \mE)|V\rangle&=&|V\rangle. \label{V}
\end{eqnarray}
\citet{derrida1993exact} give a detailed proof   for the case when $q=0$, and sketch the proof for the general $q$.
\citet[Section III]{sandow1994partially}  proves the invariance of the probability
measure determined by \eqref{MatrixSolution} for arbitrary $q$.

Formula \eqref{MatrixSolution} offers the flexibility of choosing convenient vectors and matrices.
Refs. \cite{derrida1993exact,derrida2003exact,enaud2004large,sasamoto1999one,uchiyama2004asymmetric} show such choices for various
 ranges of the parameters,
  and use the explicit representations to derive   properties of the ASEP.
Our goal is to derive a representation for the left hand side of \eqref{MatrixSolution} in terms of
moments of an auxiliary Markov process, called quadratic harness. We use
this representation to study large deviations principle for the average
occupation as well as  mean values of occupation sites. In particular, we
give simple derivations of some earlier known properties of ASEP.

Note that   representation \eqref{MatrixSolution}   cannot hold when $\<W|V\>=0$, and \citet[page 1384]{essler1996representations} point out that  this happens
when $\alpha\beta=\gamma\delta$.
 They also point out the importance of a more general requirement $\alpha\beta-q^k\gamma\delta\ne0$ for $k=0,1,\dots$.
 Our Markov process representation requires additional restrictions on the parameters, which in particular imply $\gamma\delta<\alpha\beta$.

\subsection{Quadratic harnesses}\label{AMP}
In this section we introduce an auxiliary Markov processes, called quadratic harness, defined on another probability space
which is unrelated to the probability space $(\{0,1\}^N,\pi_N)$ where the invariant measure $\pi_N$ is described through \eqref{MatrixSolution}.
To make the distinction more transparent,
we denote by $\EE(X)$ the expected value of a random variable $X$ on this probability space. Our %
Markov process is different than  the
 Markov process on
permutation tableaux which was associated with ASEP  in \cite{corteel2007markov}, and the way ASEP and the process we introduce are tied together is also different than the relation exhibited in \cite{corteel2007markov}.
 However, there is a common thread of the Askey-Wilson polynomials in the background.

A square integrable (real) stochastic process $\mathbf{X}=(X_t)_{t\in I}$ is called a  quadratic harness on an interval $I$ if for $ s<t<u$ we have
$$
\EE(X_t|\mathcal{F}_{s,u})=\tfrac{u-t}{u-s}X_s+\tfrac{t-s}{u-s}X_u,
$$
and
$$
\EE(X_t^2|\mathcal{F}_{s,u})=A_{tsu}X_s^2+B_{tsu}X_sX_t+C_{tsu}X_u^2+D_{tsu}X_s+E_{tsu}X_u+F_{tsu},
$$
where $(\mathcal{F}_{s,u})_{ s<u}$ is a past-future filtration of the process $\mathbf{X}$ and $A_{tsu},\ldots,F_{tsu}$
are non-random functions of $s,t,u$. We will say that  $\mathbf{X}=(X_t)_{t\in I}$
is a  {\em standard} quadratic harness if    $\EE\,X_t=0$, $\EE\,X_sX_t=s$, for $s<t$.
It appears (see  \cite{Bryc-Matysiak-Wesolowski-04}) that under mild technical
assumptions these conditions uniquely determine $\mathbf{X}$, which typically is a (non-homogeneous) Markov process.
Moreover, the functions $A_{tsu},\ldots,F_{tsu}$ are defined in terms of five numerical constants $\theta,\eta\in\mathbb{R}$,
 $\tau,\sigma\ge 0$, $%
 q_0\le 1+2\sqrt{\tau\sigma}$ and, consequently,
$$
\mathrm{Var}(X_t|\mathcal{F}_{s,u})=\tfrac{(u-t)(t-s)}{u(1+\sigma s)+\tau-\gamma s} K\left(\tfrac{X_u-X_s}{u-s}, \tfrac{\tfrac{1}{s}X_s-\tfrac{1}{u}X_u}{\tfrac{1}{s}-\tfrac{1}{u}} \right)
$$
with
$$K(x,y)=1+\theta x+ \eta y+ \tau x^2+ \sigma y^2-(1-q_0) xy.$$
That is, typically, a standard quadratic harness is determined by the constants $\theta,\eta,\tau,\sigma,q_0%
$,
and thus we write $\mathbf{X}\sim \mathrm{QH}(\theta,\eta,\tau,\sigma,%
q_0)$.
The family of standard quadratic harnesses includes affine transformations of such  processes as: Wiener process -
$\mathrm{QH}(0,0,0,0,1)$, Poisson process - $\mathrm{QH}(\theta,0,0,0,1)$,
 $\theta>0$ and contains the whole class of L\'evy-Meixner processes - $\mathrm{QH}(\theta,0,\tau,0,1)$ (see e.g. %
\cite{schoutens2000stochastic}). It contains also classical versions of free Brownian motion - $\mathrm{QH}(0,0,0,0,0)$ (see
\cite{biane98processes}), free Poisson process - $\mathrm{QH}(\theta,0,0,0,0)$ and the whole class of free L\'evy-Meixner processes - $\mathrm{QH}(\theta,0,\tau,0,0)$ (see, e.g.
\cite{Anshelevich2003}). Other examples of quadratic harnesses are quantum Bessel process - $\mathrm{QH}(\theta,\eta,0,0,1)$ (see  %
\cite{biane1996quelques},  %
\cite{Bryc-Wesolowski-05}, or \cite{Matysiak-Swieca-2015}), and a more general bi-Poisson process - $\mathrm{QH}(\theta,\eta,0,0,q)$ (see   %
\cite{Bryc-Matysiak-Wesolowski-04b}), or the $q$-Brownian motion - $\mathrm{QH}(0,0,0,0,q)$ (see, e.g. %
\cite{bozejko97qGaussian}).

 In
\citet{Bryc-Wesolowski-08} we use the orthogonality measures of Askey-Wilson polynomials to construct a large family of quadratic harnesses.  We recall it with some details now since this approach will be used in order to connect quadratic harnesses with the ASEPs.

Let $q\in(-1,1)$ and  $a,b,c,d\in \mathbb{C}$ be such that $abcd, abcdq, ab,abq\not\in[1,\infty)$. Following \cite[(1.20)]{Bryc-Wesolowski-08} we are interested in
 the family of Askey-Wilson polynomials %
defined through a three-term recurrence relation
\begin{equation}
  \label{EQ:bar-q}
  2x\bar{w}_n(x)=\bar{A}_n\bar{w}_{n+1}(x)+B_n\bar{w}_n(x)+\bar{C}_n\bar{w}_{n-1}(x),\qquad n\ge 0,
\end{equation}
with $\bar w_{-1}\equiv 0$, $\bar w_0\equiv 1$ and
$$\bar{A}_n=\tfrac{(1-abcdq^{n-1})(1-abq^n)}{(1-abcdq^{2n-1})(1-abcdq^{2n})},$$
\begin{multline*}
 B_n=a+\tfrac{1}{a}-\tfrac{(1-abcdq^{n-1})(1-abq^n)(1-acq^n)(1-adq^n)}{a(1-abcdq^{2n-1})(1-abcdq^{2n})} \\
 -a\tfrac{(1-q^n)(1-bcq^{n-1})(1-bdq^{n-1})(1-cdq^{n-1})}{(1-abcdq^{2n-2})(1-abcdq^{2n-1})},
\end{multline*}
$$ \bar{C}_n=\tfrac{(1-q^n)(1-acq^{n-1})(1-adq^{n-1})(1-bcq^{n-1})(1-bdq^{n-1})(1-cdq^{n-1})}{(1-abcdq^{2n-2})(1-abcdq^{2n-1})}.$$
(The bars are solely for consistency with notation in \cite{Bryc-Wesolowski-08}.) When   $\displaystyle\prod_{k=0}^n \bar A_k\bar C_{k+1}\geq 0$
for all $n$, the Askey-Wilson polynomials are orthogonal with respect to unique compactly supported
 probability measure $\nu(dx;a,b,c,d,q)$, which  is called the Askey-Wilson distribution. It is not easy to describe in general for what choices of parameters
$a,b,c,d,q$ such a measure exists.  In
\citet[Lemma 3.1]{Bryc-Wesolowski-08} sufficient conditions for existence were given in terms of two integers:
$m_1=\#(\{ab,ac,ad,bc,bd,cd\}\cap [1,\infty))$  and $m_2=\#(\{abq,acq,adq,bcq,bdq,cdq\}\cap [1,\infty))$.
For $a,b,c,d$ which are either real or come in complex conjugate pairs and such that $abcd<1$ and $abcdq<1$ the Askey-Wilson distribution exists in the following cases:
(i) $q\ge 0$ and $m_1=0$; (ii) $q<0$ and $m_1=m_2=0$; (iii) $q\ge 0$, $m_1=2$; (iv) $q<0$, $m_1=2$, $m_2=0$; (v) $q<0$, $m_1=0$, $m_2=2$.
In particular, in Section  \ref{Sect:semininf} we need to consider  $ad=1$ and $ac>1$, which falls under case (iii), with degenerated measure $\nu$   concentrated at the root of $\bar w_1(x)$.

The Askey-Wilson measure has the form
$$
\nu(dx;a,b,c,d,q)=f(x;a,b,c,d)I(|x|<1)dx+\sum_{y\in F(a,b,c,d,q)}\,p(y)\delta_y(dx),
$$
where the density
\begin{equation}
  \label{f-density}
  f(x;a,b,c,d,q)=\tfrac{(q,ab,ac,ad,bc,bd,cd)_{\infty}}{2\pi(abcd)_{\infty}\sqrt{1-x^2}}\,\left|\tfrac{\left(e^{2i\theta}\right)_{\infty}}
{\left(ae^{i\theta},be^{i\theta},ce^{i\theta},de^{i\theta}\right)_{\infty}}\right|^2,
\end{equation}
with $\cos\,\theta=x$, may be sub-probabilistic or  zero  for some values of the parameters. The set $F=F(a,b,c,d,q)$ is a finite or
empty set of atoms of $\nu(dx;a,b,c,d,q)$. If $\alpha\in\{a,b,c,d\}$ is such that $|\alpha|> 1$ then $\alpha$ is real and
generates atoms. E.g.,
 if $|a|> 1$ then it generates the atoms
\begin{equation}
  \label{aatoms}
  x_j=(aq^j+1/(aq^j))/2 \mbox{ for $j=0,1,\dots$ such that $|aq^j|\ge 1$}
\end{equation} and the corresponding masses are
\begin{multline*}
 p(x_0)=\tfrac{(a^{-2},bc,bd,cd)_{\infty}}{(b/a,c/a,d/a,abcd)_{\infty}},\; \\
p(x_j)=p(x_0)\tfrac{(a^2,ab,ac,ad)_j\,(1-a^2q^{2j})}{(q,qa/b,qa/c,qa/d)_{j}(1-a^2)}\,\left(\tfrac{q}{abcd}\right)^j,\; j\ge 1.
\end{multline*}
In the formulas above, for complex $\alpha$ and $|q|<1$ we used  $q$-Pochhammer symbol
$(\alpha)_n:=(\alpha;q)_n=\prod_{j=0}^{n-1}\,(1-\alpha q^j)$, $1\le n\le \infty$, $(\alpha)_0=1$. Moreover,  $(a_1,\ldots,a_k)_n:=(a_1)_n\ldots(a_k)_n$, $0\le n\le \infty$.

We now use the Askey-Wilson measures $\nu(dx;a,b,c,d,q)$ to construct a Markov process which will depend on parameters
$A,B,C,D,q$. As previously, we take  $q\in(-1,1)$ and fix $A,B,C,D$ which are either real or $(A,B)$ or $(C,D)$ are complex
conjugate pairs with $ABCD<1$, $qABCD<1$ and such that
\begin{equation}
  \label{ABCD-oryginal}
  AC,AD,BC,BD,qAC,qAD,qBC,qBD\in\mathbb{C}\setminus[1,\infty).
\end{equation}
The Askey-Wilson stochastic process $\mathbf{Y}=(Y_t)_{t\in I}$, where
\begin{equation}
  \label{EQ:I}
  I=\left[\max\{0,CD,qCD\},\,\tfrac{1}{\max\{0,AB,qAB\}}\right),
\end{equation} is defined as the Markov process with marginal
distributions
\begin{equation}
  \label{Y-univariate}
  \pi_t(dx)=\nu(dx;A\sqrt{t},B\sqrt{t},C/\sqrt{t},D/\sqrt{t},q),\quad t\in I,
\end{equation} with compact support $U_t$, and the transition probabilities
\begin{multline}
\label{Y-transitions}
   P_{s,t}(x,dy)=\nu(dy;A\sqrt{t},B\sqrt{t},\sqrt{s/t}(x+\sqrt{x^2-1}),\sqrt{s/t}(x-\sqrt{x^2-1})), \\
   \quad  s<t,\;s,t\in I,\;x\in U_s.
\end{multline}
The marginal distribution $\pi_t(dx)$ may have atoms at points
\begin{multline}
\label{Yatoms}
 \frac12\left(B\sqrt{t}q^j +\frac{1}{B\sqrt{t}q^j}\right),\;\frac12\left(\frac{Dq^j}{\sqrt{t}} +\frac{\sqrt{t}}{Dq^j}\right),\;
\\
 \frac12\left(A\sqrt{t}q^j +\frac{1}{A\sqrt{t}q^j}\right),\;  \frac12\left(\frac{Cq^j}{\sqrt{t}} +\frac{\sqrt{t}}{Cq^j}\right)  \;.
\end{multline}
These formulas  are recalculated from \eqref{aatoms} and appear explicitly in \cite[(3.7) and (3.8)]{Bryc-Wesolowski-08}.
(Transition probabilities $P_{s,t}(x,dy)$  also may have atoms.)

To construct a standard quadratic harness from the Askey-Wilson process we introduced in \cite[(2.22)]{Bryc-Wesolowski-08}
 an intermediate (nonstandard) quadratic harness  $\mathbf{Z}$ on $I$, which will be important for our further considerations in this paper. The process $\mathbf{Z}=(Z_t)_{t\in I}$ is defined by \begin{equation}
   \label{JW-Z} Z_t=\tfrac{2\sqrt{t}}{\sqrt{1-q}}\,Y_t,  \mbox{$t\in I$}.
 \end{equation} It is a Markov process with the following properties:
 \begin{itemize}
   \item[(i)] $(Z_t)$ has a sequence of orthogonal martingale polynomials $$\langle\mathbf{r}_t(x)|=[r_0(x;t), r_1(x;t),\dots],$$
    where
\begin{equation}
  \label{w2r} r_n(x;t)=t^{n/2}\bar{w}_n\left(\tfrac{\sqrt{1-q}}{2\sqrt{t}}x;A\sqrt{t},B\sqrt{t},C/\sqrt{t},D/\sqrt{t},q\right), \mbox{
 $n=0,1,\dots$},
\end{equation}
    with  $\bar{w}_n(x)=\bar{w}_n(x;a,b,c,d,q)$ defined in \eqref{EQ:bar-q}. That is, $r_n(x;t)$ is a polynomial of degree $n$
    in variable $x$ such that:
    \begin{enumerate}
      \item $r_0(x;t)=1$;
      \item $\EE\left(r_n(Z_t;t)r_m(Z_t;t)\right)=0$ for $m\ne n$;
      \item $\EE(r_n(Z_t;t)|Z_s)=\EE(r_n(Z_t;t)|\sigma(Z_v:v\leq s))=r_n(Z_s;s)$ for $s\leq t$.
    \end{enumerate}
    \item[(ii)]  The Jacobi matrix of the orthogonal polynomials $\{r_n(\cdot;t)\}$ is a tridiagonal infinite matrix which
    depends linearly on parameter $t$. We will write it  as $t\mx+\my$, and we will write the so called ``three step recurrence" in vector notation as
    \begin{equation}
      \label{three-step} x \langle\mathbf{r}_t(x)| = \langle\mathbf{r}_t(x)|(t\mx+\my).
    \end{equation}
    This is recalculated from the recurrence \eqref{EQ:bar-q}, and appears in  \cite[page 1244]{Bryc-Wesolowski-08}.
    \item[(iii)]  The Jacobi matrix is determined uniquely from the $q$-commutation equation
    \begin{equation}
      \label{q-comm}
      \mx \my - q \my \mx =\mI
    \end{equation}
    and the initial conditions  %
    \begin{equation}
      \label{ini-cond}
      \mx=\left[\begin{matrix}
        \gamma_0 & \eps_1 & 0 &\dots \\
        \alpha_0 &  &  & \\
        0 &  & & \\
        \vdots & & & \\
      \end{matrix}\right], \quad \my=\left[\begin{matrix}
        \delta_0 & \varphi_1 & 0 &\dots \\
        \beta_0 &  &  & \\
        0 &  &  & \\
        \vdots & & & \\
      \end{matrix}\right]\,,
    \end{equation}
    where $\alpha_0,\beta_0,\gamma_0,\delta_0,\eps_1,\varphi_1$ depend on parameters $A,B,C,D,q$ through formulas in
    \cite[page 1243]{Bryc-Wesolowski-08}; these formulas are reproduced in the proof of Theorem \ref{T-rep} below, see e.g.
    \eqref{gamma0}.
 \end{itemize}
 \citet{Bryc-Wesolowski-08} show that  $(Z_t)_{t\in I}$ is a Markov process
 with harness property
 \begin{equation}\label{QH}
  \EE(Z_t|Z_s,Z_u)=\frac{u-t}{u-s}Z_s+\frac{t-s}{u-s}Z_u
 \end{equation}
 and quadratic conditional variances under a two-sided conditioning.
 However, to obtain a standard quadratic harness an adjustment is necessary in order to fulfill the expectations
 and covariances requirements while preserving linearity of conditional expectations and quadratic form of conditional variances. This adjustment uses an invertible M\"obius transformation
  $T(t)=\tfrac{t+CD}{1+ABt}$ and is of the form
 \begin{multline}
 \label{EQ:Z2X}
 X_t=\tfrac{\sqrt{1-q}(1-ABt)Z_{T(t)}-(A+B)t-(C+D)}{\sqrt{(1-q)(1-AC)(1-AD)(1-BC)(1-BD)}}\,\sqrt{1-qABCD},
 \\  t\in J:=T^{-1}(I),
 \end{multline}
 see  \cite[formula (2.28)]{Bryc-Wesolowski-08}. \comment{"coma" in formula, and added "period"}

 Ref. \cite[Theorem 1.1]{Bryc-Wesolowski-08}
 states that $\mathbf{X}=(X_t)_{t\in J}$ is a quadratic harness $\mathrm{QH}(\theta,\eta,\tau,\sigma,
 q_0)$ with parameters
\begin{equation}
  \label{eq:theta}
  \theta=-\tfrac{((C+D)(1+ABCD)-2CD(A+B))\sqrt{1-q}}{\sqrt{(1-AC)(1-AD)(1-BC)(1-BD)(1-qABCD)}},
\end{equation}
\begin{equation}\label{eq:eta}
\eta=-\tfrac{((A+B)(1+ABCD)-2AB(C+D))\sqrt{1-q}}{\sqrt{(1-AC)(1-AD)(1-BC)(1-BD)(1-qABCD)}},
\end{equation}
\begin{equation}
  \label{eq:gamma} \tau=\tfrac{CD(1-q)}{1-qABCD},\quad \sigma=\tfrac{AB(1-q)}{1-qABCD},\quad %
  q_0=\tfrac{q-ABCD}{1-qABCD}.
\end{equation}

Ref. \cite[Theorem 1.2]{Bryc-Wesolowski-08} assures that Markov process $(Z_t)_{t\geq 0}$ exists when conditions \eqref{ABCD-oryginal} hold.
In this paper we
will only need real-valued parameters $A\geq 0,-1<B\leq 0,C\geq 0,-1<D\leq 0$, and $0\leq q<1$. Specialized to this case, \eqref{ABCD-oryginal} reduces to the condition $AC<1$.

\section{Representing  a matrix solution by the auxiliary Markov process}\label{sect:AMP}
We now return  to the study of invariant measure of ASEP with parameters $\alpha>0$, $\beta>0$, $\gamma\geq 0$, $\delta\geq 0$, $0\leq q<1$.
Our goal is to associate an appropriate quadratic harness $(Z_t)$ with the ASEP. Denote
 \begin{equation}\label{kappa}
\kappa_{\pm}(u,v)=\frac{1-q-u+v\pm\sqrt{(1-q-u+v)^2+4u v}}{2u} \,.
\end{equation}

 Let $(Z_t)_{t\geq 0}$ be the Markov process from \eqref{JW-Z} in Section \ref{AMP} with parameters %
\begin{equation}
  \label{ABCD}
  A=\kappa_{+}(\beta,\delta), \; B= \kappa_{-}(\beta,\delta), \; C=\kappa_{+}(\alpha,\gamma),\;D=\kappa_{-}(\alpha,\gamma).
\end{equation}
Notice that $A,B,C,D$ are real, $ABCD=\gamma\delta/(\alpha\beta)$, and $A,C\geq 0$ while $-1<B,D\leq 0$.
Indeed, it is clear that $\kappa_-(u,v)\leq 0$ when $u>0$, $v\geq 0$. Inequality  $\kappa_-(u,v)>-1$ is equivalent to
$1-q+u+v>\sqrt{(1-q-u+v)^2+4uv}$.
Since $q<1$ and $u,v\geq0$, the left hand side is positive. Upon squaring both sides we see that $\kappa_-(u,v)>-1$ is equivalent to
$4u(1-q)>0$.

In particular, $BD<1$, so $(Z_t)_{t\in I}$ is well defined when $AC<1$ and from
\eqref{EQ:I} we get $I=(0,\infty)$.

Our main result is the following representation of the joint
 generating function of the invariant measure of the ASEP in terms
 of the stochastic process $(Z_t)$.

\begin{theorem}
  \label{T-rep}
Suppose that the parameters of ASEP are such that
  $AC<1$. Then for $0<t_1\leq t_2\leq \dots\leq t_N$ the joint generating function of the invariant measure of the ASEP  is
  \begin{equation}\label{QHsolution}
\left\langle \prod_{j=1}^N t_j^{\tau_j}\right\rangle_N=\frac{\EE\left(\prod_{j=1}^N(1+t_j+\sqrt{1-q}\,Z_{t_j})\right)}{\EE\left((2+\sqrt{1-q}\,Z_1)^N\right)}.
  \end{equation}
\end{theorem}
Since the left hand side is a polynomial in $t_1,\dots,t_N$, it is clear that formula \eqref{QHsolution} determines
 $\langle \prod_{j=1}^N t_j^{\tau_j}\rangle_N$  uniquely for all $t_1,\dots,t_N$.

The proof of Theorem \ref{T-rep} appears  in Section \ref{Sect:PTR}. We conclude this section with  remarks.
 \begin{remark}\label{Rem:interpret}
Assumption $AC<1$  has natural physical interpretation -- in terminology of \cite{derrida2003exact}
 it identifies the ``fan region" of ASEP.  Since $0\leq BD<1$, we have
$$\frac{\gamma\delta}{\alpha\beta}=ABCD<1,$$
so the conditions \cite[(58)]{essler1996representations} for the existence of a matrix representation are automatically
satisfied.
\end{remark}

\begin{remark}
Suppose that $(Z_t)$ is a Markov process from Section \ref{AMP} with $0\leq q<1$ and with parameters $A,B,C,D$ which
  satisfy \eqref{ABCD-oryginal}.  If in addition \ $A,B,C,D$ are real, $q\geq 0$,
$(1+A)(1+B)>0$, $(1+C)(1+D)>0$,  and $AB\leq 0$,  $CD\leq 0$,  then the right hand side of \eqref{QHsolution} determines the invariant measure of an ASEP with
parameters
\begin{multline}
  \label{Jacek}
 \alpha=\frac{1-q}{(1+C)(1+D)}, \;   \beta=\frac{1-q}{(1+A)(1+B)},\; \gamma=\frac{-(1-q) CD}{(1+C)(1+D)},\\ \delta=\frac{-(1-q)AB}{(1+A)(1+B)}.
\end{multline}

\end{remark}

\begin{remark}
A special case of formula \eqref{QHsolution} relates $\<t^{\sum_{j=1}^n\tau_j}\>_N$ to the $N$-th moment of  $(1+t+\sqrt{1-q}Z_t)$, i.e., to the moments of an Askey-Wilson law.
Therefore  exact formula for   $\EE\left((2+\sqrt{1-q}Z_1)^N\right)/(1-q)^N$ when $\gamma=\delta=0$ can be read from \cite[Equation (57)]{blythe2000exact}.
 \citet[Theorem 1.3.1]{josuat2011combinatorics} and \cite[Theorem 6.1]{josuat2011rook} used combinatorial techniques to generalize  this to the exact formula for
\comment{added missing 1}  $\EE\left((1+t+\sqrt{1-q}Z_t)^N\right)/(1-q)^N$.
  For a more analytic approach, see  \cite{szablowski2015moments}.
Formulas for moments of  more general Askey-Wilson laws appear in
\cite{corteel2012formulae,kim2014moments}.

\end{remark}

\begin{remark}
The right hand side of \eqref{QHsolution} may define a probability generating function of a probability measure for other values
  of parameters than those specified in the assumptions of Theorem \ref{T-rep}.
  In Theorem \ref{T-semiinfinite} we show
  that a semi-infinite ASEP is related to \eqref{QHsolution} with parameters that do not correspond to an ASEP.
  \citet{johnston2012pasep} consider another ``non-physical" case $q=-1$.%
\end{remark}

\subsection{Proof of Theorem \ref{T-rep}}\label{Sect:PTR} %
Our motivation and starting point is Ref. \cite{uchiyama2004asymmetric} which represents matrices $\mE,\mD$
by the Jacobi matrices of the Askey-Wilson polynomials, with vectors $\langle W|=[1,0,0,\dots]$ and $|V\rangle=[1,0,0,\dots]^T$.
Ref. \cite{uchiyama2004asymmetric} then uses these matrices to give  the integral formulae for the partition function and for the $m$-point correlation function
$\langle \tau_{j_1}\tau_{j_2}\dots\tau_{j_m}\rangle$.
These integral formulas inspired our search for representation in terms of the quadratic harnesses.

We begin with a slight variation of \cite[formula (4.5)]{uchiyama2004asymmetric}. We use matrices $\mx$, $\my$ from
\eqref{three-step}   to rewrite $\mE,\mD$
 as
\begin{equation}
  \label{J2C}
\mE= \frac{1}{1-q}\mI+\frac{1}{\sqrt{1-q}}\,\my, \quad \mD=\frac{1}{1-q}\mI+\frac{1}{\sqrt{1-q}}\,\mx\,,
\end{equation}
compare \cite[formulas (13) and (14)]{blythe2000exact}. (This is simply a rescaling by the factor $1/\sqrt{1-q}$ of the two
operators in \cite[formula (4.5)]{uchiyama2004asymmetric}.) %
Then, just like in
\cite{uchiyama2004asymmetric}, formulas \eqref{q-comm} and \eqref{q-comm-Derrida} are equivalent.

Next, we use \eqref{ini-cond}  to ensure that \eqref{V} and \eqref{W} hold with  vectors $\langle W|=[1,0,0,\dots]$ and
$|V\rangle=[1,0,0,\dots]^T$. In order for \eqref{V} to hold,  in  \eqref{ini-cond} we must have
\begin{equation}
  \label{corner10}
  \beta\alpha_0-\delta\beta_0=0
\end{equation} and
\begin{equation}
  \label{corner1}
  \frac{\beta-\delta}{1-q}+\frac{\beta\gamma_0-\delta \delta_0}{\sqrt{1-q}}=1.
\end{equation}
Since according to  \cite[page 1243]{Bryc-Wesolowski-08}, $\alpha_0=-AB \beta_0$,  equation \eqref{corner10} gives
$AB=-\delta/\beta$. The expressions for $\gamma_n$ and $\delta_n$  presented in  \cite[page 1243]{Bryc-Wesolowski-08}  for
$n=0$   give
\begin{equation}
  \label{gamma0}
  \gamma_0=\frac{A+B-AB(C+D)}{\sqrt{1-q}(1-ABCD)}, \;
  \delta_0=\frac{C+D-CD(A+B)}{\sqrt{1-q}(1-ABCD)}.
\end{equation}
We note that condition $ABCD<1$ implies that $\alpha\beta-\gamma\delta\ne 0$, so inserting these expressions into \eqref{corner1},    we get
$(1+C+D)\alpha=1-q+\gamma$.

Similarly, in order for \eqref{W} to hold, we must have $\alpha\varphi_1-\gamma\eps_1=0$,
 which  in view of relation $\varphi_1=-CD\eps_1$  \cite[page 1243]{Bryc-Wesolowski-08}
gives $CD=-\gamma/\alpha$. Finally, we need to ensure that
$$
\frac{\alpha-\gamma}{1-q}+\frac{\alpha \delta_0-\gamma\gamma_0}{\sqrt{1-q}}=1.
$$
Using \eqref{gamma0} and $\alpha\beta\ne\gamma\delta$, this simplifies to $(1+A+B)\beta=1-q+\delta$. The system of four equations for $A,B,C,D$ decouples,
 and we get a  pair of one-variable quadratic equations for the unknowns $A,C$,  with solutions \eqref{ABCD}.

We now choose $t_1\leq t_2\leq\dots\leq t_N$ and, after including the factor $ (1-q)^N$ into the normalizing constant, we  rewrite
 the unnormalized matrix expression on the
right hand side of \eqref{MatrixSolution} as
\begin{multline}
  \label{QH-step1}
\langle W| \mathop{\stackrel{\to}{\prod}}_{j=1}^N(\mE+t_j\mD)|V\rangle
\\
=
 \frac{1}{(1-q)^N}\langle W|  \mathop{\stackrel{\to}{\prod}}_{j=1}^N((1+t_j)\mI+\sqrt{1-q}(t_j\mx+\my))|V\rangle:= \frac{1}{(1-q)^N}\Pi.
\end{multline}
Since polynomials $\{r_n(x;t)\}$ are orthogonal and $r_0(x;t)=1$, we have $\langle W|=\EE(\langle \mathbf{r}_{t_1}(Z_{t_1})|)$ which gives
$$
\Pi
 = \EE\left(\Big\langle \mathbf{r}_{t_1}(Z_{t_1})\Big| \mathop{\stackrel{\to}{\prod}}_{j=1}^N((1+t_j)\mI+\sqrt{1-q}(t_j\mx+\my))\big|V\Big\rangle\right).
$$
By \eqref{three-step},
$$\big\langle \mathbf{r}_{t}(Z_{t})\big|  \left((1+t)\mI+\sqrt{1-q}(t\mx+\my)\right)= (1+t+\sqrt{1-q}Z_{t}) \big\langle \mathbf{r}_{t}(Z_{t})\big|,$$
so
\begin{multline*}
\Pi
 = \EE\Big(\left(1+t_1+\sqrt{1-q}Z_{t_1}\right) \\ \times \big\langle \mathbf{r}_{t_1}(Z_{t_1})\big|  \mathop{\stackrel{\to}{\prod}}_{j=2}^N((1+t_j)\mI+\sqrt{1-q}(t_j\mx+\my))\big|V\big\rangle\Big).
\end{multline*}
We now use this formula,   martingale property of $\{r_n(Z_t;t)\}$ and
 standard properties of  conditional expectations. We get
\begin{multline*}
\Pi=\EE\Big((1+t_1+\sqrt{1-q}Z_{t_1}) \\
\times\EE(\langle \mathbf{r}_{t_2}(Z_{t_2})|  \mathop{\stackrel{\to}{\prod}}_{j=2}^N((1+t_j)\mI+\sqrt{1-q}(t_j\mx+\my))|V\rangle|Z_{t_1})\Big)
\\
=\EE\Big((1+t_1+\sqrt{1-q}Z_{t_1}) \\ \times \langle \mathbf{r}_{t_2}(Z_{t_2})|  \mathop{\stackrel{\to}{\prod}}_{j=2}^N((1+t_j)\mI+\sqrt{1-q}(t_j\mx+\my))|V\rangle \Big)
\\
=\EE\Big(  \prod_{i=1}^2(1+t_i+\sqrt{1-q}Z_{t_i}) \langle \mathbf{r}_{t_2}(Z_{t_2})|  \mathop{\stackrel{\to}{\prod}}_{j=3}^N((1+t_j)\mI+\sqrt{1-q}(t_j\mx+\my))|V\rangle\Big)
\\
=\EE\Big(  \prod_{i=1}^2(1+t_i+\sqrt{1-q}Z_{t_i})
 \\ \times  \EE\Big(\big\langle \mathbf{r}_{t_3}(Z_{t_3})\big|   \mathop{\stackrel{\to}{\prod}}_{j=3}^N((1+t_j)\mI+\sqrt{1-q}(t_j\mx+\my))\big|V\big\rangle\Big|Z_{t_1},Z_{t_2}\Big)\Big).
\end{multline*}
Recurrently, %
we get
\begin{equation}\label{Pi}
\Pi =
\EE\left(\prod_{i=1}^N(1+t_i+\sqrt{1-q}Z_{t_i})  \langle \mathbf{r}_{t_N}(Z_{t_N})|V\rangle\right).
\end{equation}
Recall now that $r_0(x;t)=1$ and $|V\rangle=[1,0,0,\dots]^T$, so $\langle \mathbf{r}_{t_N}(Z_{t_N})|V\rangle=r_0(Z_{t_N};t_N)=1$. Thus \eqref{QH-step1} and \eqref{Pi} give
$$\big\langle W\big|  \mathop{\stackrel{\to}{\prod}}_{j=1}^N(\mE+t_j\mD)\big|V\big\rangle= \frac{1}{(1-q)^N}\EE\left(\prod_{j=1}^N \left(1+t_j+\sqrt{1-q}Z_{t_j}\right)\right).$$
By \eqref{MatrixSolution}, this proves \eqref{QHsolution}.

\begin{remark}
  \label{Rem:K} The normalizing constant $K_N$ from  \eqref{KN} enters many exact formulas in the literature.  Of course, the (new) normalizing constant
 is just the value of the expression $\EE\left(\prod_{j=1}^N \left(1+t_j+\sqrt{1-q}Z_{t_j}\right)\right)$
at $t_j=1$. Our proof used matrix representation \eqref{MatrixSolution} with $\<W|V\>=1$; by homogeneity, in general  the new normalization is related
to $K_N$ by formula
  \begin{equation}\label{recalc-K2U}
  \EE\left((2+\sqrt{1-q}\,Z_1)^N\right)= \frac{(1-q)^N}{\langle W|V\rangle} K_N.
  \end{equation}

\end{remark}

\section{Application to large deviations}\label{Sect:LD}
This section shows how one can use   representation \eqref{QHsolution} to prove the large deviations principle  for $\tfrac1N\sum_{j=1}^N\tau_j$.

For $q=0$, $\alpha,\beta>1$ and $\gamma=\delta=0$ this result was announced in \cite[(30)]{derrida2001free}. LDP for general
$q$ appears in \cite[Section 3.5]{derrida2003exact}, as an application to a more general path-level LDP. The  general case with
$\gamma,\delta\geq 0$ is described in the introduction to \cite{enaud2004large}, but it seems that no proof is available in the literature: in \cite{derrida2003exact} the authors write only: {\em the calculations in the case of nonzero $\gamma$ or $\delta$ are
more complicated, and for that reason we limit our analysis in the present
paper to the case $\gamma=\delta=0$}. Our proof is a good illustration of the use of %
Theorem \ref{T-rep} to get rigorous proofs
when the parameters of the ASEP are within its range of applicability.

For an ASEP with parameters $\alpha,\beta,\gamma,\delta,q$, define $A,B,C,D$ by \eqref{ABCD} and let
\begin{eqnarray}
  \label{rho}
  \rho_0&=&\frac{1}{1+C}=%
  \frac{1-q+\alpha+\gamma-\sqrt{(1-q-\alpha+\gamma)^2+4\alpha\gamma}}{2(1-q)} ,\label{rho-a}\\
   \rho_1&=&\frac{A}{1+A}=%
    \frac{1-q-\beta-\delta+
   \sqrt{(1-q-\beta+\delta)^2+4\beta\delta}}{2(1-q)}. \label{rho-b}
\end{eqnarray}
(These are \cite[(1.1a) and (1.1b)]{enaud2004large} and
 \cite[formula (100)]{essler1996representations}.) Assumption $AC<1$
 is equivalent to $\rho_0>\rho_1$, and this is the only case that we analyze here.

 Next, we define
\begin{multline}
  \label{L0}
  \LL_0=- \sup_{\rho_1\leq \rho\leq \rho_0} \log(\rho(1-\rho))  \\=\begin{cases}
    - \log(\rho_0(1-\rho_0))= 2\log (\sqrt{C}+\frac{1}{\sqrt{C}}) & \mbox{if $C>1$, i.e., $\rho_0<1/2$},\\
     - \log(\rho_1(1-\rho_1))= 2\log (\sqrt{A}+\frac{1}{\sqrt{A}}) & \mbox{if $A>1$, i.e., $\rho_1>1/2$},\\
     \log 4 &\mbox{if $A,C<1$, i.e., $\tfrac12\in[\rho_1,\rho_0]$}. \\
   \end{cases}
\end{multline}
compare \cite[expression (1.5)]{derrida2003exact} or \cite[expression (1.7)]{enaud2004large}. The rate function depends on the
relative entropy of Bernoulli distributions:
\begin{equation}
  \label{h}
  h(x|p)=x \log \tfrac{x}{p}+(1-x)\log\tfrac{1-x}{1-p} \mbox{ for $x\in[0,1], p\in(0,1)$}.
\end{equation}
\begin{theorem} \label{T-LDP} Suppose that the parameters of ASEP are such that $\rho_0>\rho_1$.
  Then the sequence $\{\bar\tau_N:=\frac1N\sum_{j=1}^N\tau_j\}_{N\in\NN}$ satisfies the large deviations principle with speed $N$ and the rate function
  \begin{equation}
    \label{Ir}
    \I(x)=\begin{cases}
     h(x|\rho_0) +\LL_0+\log(\rho_0(1-\rho_0))& \mbox{if $0\leq x<1-\rho_0$},\\
     2 h(x|\tfrac12)+\LL_0-\log 4 &\mbox{if $1-\rho_0\leq x\leq 1-\rho_1$} ,\\
       h(x|\rho_1) +\LL_0+\log(\rho_1(1-\rho_1))& \mbox{if $1-\rho_1<x\leq 1$},\\
      \infty &\mbox{ if $x< 0$ or $x> 1$}.
    \end{cases}
  \end{equation}
  That is, denoting by $\pi_N$ the invariant measure on $\{0,1\}^N$ with  generating function \eqref{MatrixSolution}, we have
  \begin{equation}
    \label{LDP1}
    \lim_{N\to\infty} \frac1N\log \pi_N( \bar\tau_N\in(a,b))=-\inf_{x\in(a,b)}\I(x)
  \end{equation}

  for any $0\leq a<b\leq 1$.
\end{theorem}
Recall that $\pi_N( \bar\tau_N\in(a,b))$ is an abbreviation for $\pi_N(\{\tau\in \{0,1\}^N:\,\bar{\tau}_N\in(a,b)\})$.
We  remark
that since the rate function $\I(\cdot)$ is continuous on $[0,1]$, the upper/lower bounds of the ``standard statement''  of the
large deviations principle coincide on the subintervals of $[0,1]$.

  Theorem \ref{T-LDP} formally extends \cite[Section 3.5, formula (3.12)]{derrida2003exact}  to allow (some)  $\gamma,\delta$. It also includes
the LDP announced in \cite[formula (30)]{derrida2001free}, who considered the case $\alpha,\beta>1$ and $\gamma=\delta=q=0$.
Indeed, if $\gamma=\delta=0$ and $\alpha>1-q$, $\beta>1-q$ then  \eqref{kappa} and \eqref{ABCD}
give $A=C=0$. So
 $\rho_0=1$, $\rho_1=0$ and the rate function is $I(x)=2h(x|1/2)$ for $0\leq x\leq 1$.

In particular, Theorem \ref{T-LDP} recovers   part of the ``phase diagram''  \cite{derrida1992exact,derrida1993exact,enaud2004large,schutz1993phase}:
from \eqref{Ir} and \eqref{L0} it is easy to locate the unique zero of $\I(x)$, so that $\bar\tau_N$ converges exponentially fast to
either $\rho_0$, if $\rho_0<1/2$ (the so called {\em low density phase}),
 to $\rho_1$, if $\rho_1>1/2$ (the so called {\em high density phase}), or to $1/2$, if $1/2\in [\rho_1,\rho_0]$
 (the so called {\em maximal current phase}).

\begin{proof}
The proof is based on a routine application of the  Laplace transform method from the G\"artner-Ellis Theorem  \cite[Theorem
2.3.6]{DZ2009}. That is, we   compute the limit
\begin{equation}
  \label{LL}
  \Lambda(\la)=\lim_{N\to\infty}\frac1N\log \langle \exp(N \la \bar \tau_N)\rangle_N=\lim_{N\to\infty}\frac1N\log
  \left\langle \exp( \la \sum_{j=1}^N \tau_j)\right\rangle_N
\end{equation}
for real $\la$. In view of \eqref{QHsolution}, this amounts to calculation of the limit
\begin{equation}
  \label{LL0}
  \LL(\lambda)=\lim_{N\to\infty}\frac1N \log \EE\left((1+e^\la +\sqrt{1-q} Z_{e^\la})^N \right),
\end{equation}
with $\Lambda(\la)=\LL(\lambda)-\LL(0)$.

Our goal is to show that $1+e^\la +\sqrt{1-q} Z_{e^\la}\geq 0$, so
\begin{equation}
  \label{LL-supp}
  \LL(\lambda)=\log \|1+e^\la +\sqrt{1-q}
Z_{e^\la}\|_\infty
\end{equation} depends only on the support of the single random variable  $Z_t$, and is given by  explicit formula
\eqref{L0-ans} below.

 We determine the support of $\sqrt{1-q}Z_t$  from   the
support of a closely related random variable $Y_t$ in \cite[(2.22)]{Bryc-Wesolowski-08}. From \cite[Section
3.2]{Bryc-Wesolowski-08}, we read out that the distribution of $Z_t$ has only absolutely continuous and discrete parts. The
absolutely continuous part of the law of $\sqrt{1-q}Z_t$ is supported on the interval $[-2\sqrt{t},2\sqrt{t}]$, because the
continuous part of $Y_t$ is supported on $[-1,1]$. The locations of atoms of $\sqrt{1-q}Z_t$ are calculated from
\eqref{Yatoms} and \eqref{aatoms}. Recalling that according to \eqref{ABCD}, we have $A,C\geq 0$ and $B,D\leq 0$, the atoms
appear above (i.e., to the right of) the support $[-2\sqrt{t},2\sqrt{t}]$ of the continuous part at points
$$At q^j +\frac{1}{A q^j}  \mbox{ for $t>\frac{1}{A^2 q^{2j}}$}, \; C q^j+\frac{t}{C q^j} \mbox{ for $t<C^2 q^{2j}$}, $$
and they  appear below (i.e., to the left of) the support $[-2\sqrt{t},2\sqrt{t}]$ of the continuous part  at points
$$ Bt q^j +\frac{1}{B q^j} \mbox{ for $t>\frac{1}{B^2 q^{2j}}$}, \; D q^j+\frac{t}{D q^j} \mbox{ for $t<D^2 q^{2j}$}. $$

The atoms above the support of the continuous part lie below the pair of the    atoms  that correspond to
$j=0$, and   there are no atoms above the support of the continuous part when $C^2<t<1/A^2$ (recall that this interval is
nonempty by assumption \eqref{ABCD}). Similarly, the atoms below the support of the continuous part lie above the pair of the
atoms   that correspond to $j=0$. So the support of $\sqrt{1-q}Z_t$ is contained between the pair of continuous curves. The
upper boundary is  \begin{equation}
  \label{u(t)}
  U(t)=\begin{cases}
   C+t/C &\mbox{if  $t\leq C^2$}, \\
   2\sqrt{t} & \mbox {if $C^2<t<1/A^2$},\\
   At+ 1/A & \mbox{if $t\geq 1/A^2$},
\end{cases}
\end{equation}
and the lower boundary is
\begin{equation}\label{ell(t)}
\ell(t)=\begin{cases}
   D+t/D & \mbox{ if $t\leq D^2$}, \\
   -2\sqrt{t} & \mbox{ if $D^2<t<1/B^2$},\\
   Bt+ 1/B & \mbox{ if $t\geq 1/B^2$}\,.
\end{cases}
\end{equation}	
The %
boundaries do not depend on the value of parameter $0\leq q<1$.    Fig. \ref{Fig2} illustrates how the support of $Z_t$ varies
with $t$ for $q=0$.

\begin{figure}[H]
\includegraphics[width=4in]{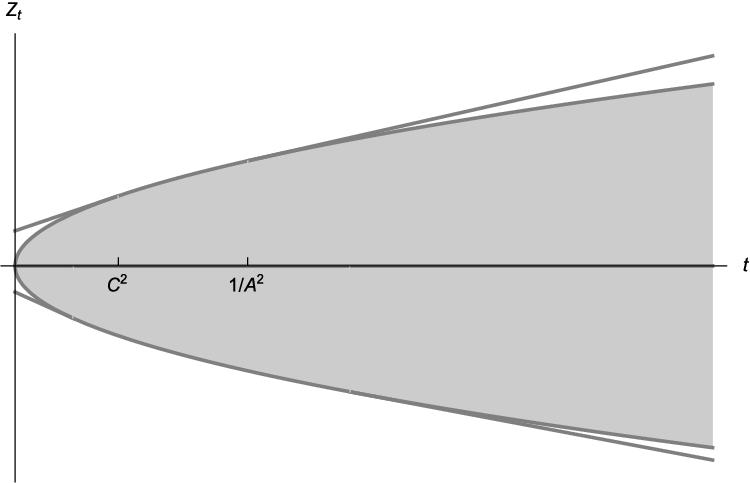}
\caption{Support of $Z_t$ for $q=0$, $A,C>0$, $B,D<0$. The shaded gray area between the parabolas represents the support
 $[-2\sqrt{t},2\sqrt{t}]$ of the continuous part of the distribution. The  four tangent
gray lines represent the locations of the atoms for   $t>0$.
\label{Fig2}
}
\end{figure}

To verify that $1+t+\sqrt{1-q}Z_t\geq 0$, we inspect  each of the cases  and verify that $1+t+\ell(t)\geq (1-\sqrt{t})^2\geq 0$:
\begin{enumerate}
  \item for $t\leq D^2$ with negative $D$, we have $1+t+D+t/D=(D+t)(1+1/D)\geq  (1-\sqrt{t})^2$ as $D\mapsto D+t/D$ with derivative $1-t/D^2\geq 0$ is increasing on $-\sqrt{t}<D<0$;

 \item for any $t$, $1+t-2\sqrt{t}=(1-\sqrt{t})^2\geq 0$;

\item  for $t\geq 1/B^2$  with negative $B$, we have $1+t+tB+1/B =(B+1)(t+1/B)\geq  (1-\sqrt{t})^2$, as $B\mapsto tB+1/B$ with derivative $t-1/B^2\geq 0$
is increasing on $-1/\sqrt{t}<B<0$.
\end{enumerate}

 Since we proved that $1+t+\sqrt{1-q}Z_t\geq 0$, we can now insert the absolute value and the limit \eqref{LL0}
becomes just the logarithm of the infinity norm   $\|1+e^\lambda+\sqrt{1-q}Z_{e^\lambda}\|_\infty$. The explicit value of the
limit can now be read out from \eqref{u(t)}. We get
\begin{equation}\label{L0-ans}
  \LL(\lambda)=\begin{cases}
\log (C+e^\la)+\log\frac{1+C}{C}  & \mbox{if  $e^\la<C^2$}, \\
2 \log(1+e^{\la/2})   &\mbox{ if $C^2<e^\la<1/A^2$} , \\
\log(1+A e^\la)+\log \frac{1+A}{A} & \mbox{ if $e^\la>1/A^2$}.
  \end{cases}
\end{equation}
Accounting for  the normalization which is needed in \eqref{LL}, we get $\Lambda(\la)=\LL(\la)-\LL(0)$. A calculation confirms
that $\LL_0:=\LL(0)$ is indeed given by \eqref{L0}.

Since the function $\la\mapsto \Lambda(\la)$ is  differentiable at all $\la$, the LDP follows with
$$
\I(x)=\sup_{\la}\{\la x-\Lambda(\la)\}=\LL_0+\sup_{\la}\{\la x-\LL(\la)\}.
$$
To derive the three explicit expressions given in \eqref{Ir}
we note that
$$\LL'(\la)=\begin{cases}
  \frac{e^\la}{C+e^\la} & \mbox{with $e^\la<C^2$ has range $(0,C/(1+C))$ },\\
  \frac{e^{\la/2}}{1+e^{\la/2}} & \mbox{with $C^2\leq e^\la\leq 1/A^2$ has range $[C/(1+C), 1/(1+A)]$ },\\
  \frac{Ae^\la}{1+Ae^\la} & \mbox{with $e^\la>1/A^2$ has range $(1/(1+A), 1)$}.
\end{cases}$$
After a calculation, this gives each of the cases listed in \eqref{Ir} for $x\in(0,1)$.
The value for $\I(x)$ with $x\leq 0 $  arises from $\la\to-\infty$ with $\LL(\la)\to \log(1+C)$.
The value for $\I(x)$ with  $x\geq 1$
arises from taking  $\la\to\infty$ with $\la x -\LL(\la)= (x-1)\la  +(\la -\LL(\la)) $ and noting that
 $\la -\LL(\la)=\log\frac{e^\la}{1+A e^\la}-\log\frac{1+A}{A}\to -\log(1+A)$ is bounded.
(The above is a routine entropy calculation that we included for completeness.)

\end{proof}

\section{Some explicit formulas} \label{Sect:miracle}
 Some ``miraculous'' explicit formulas for the average occupation $\langle \tau_k\rangle_N $ of site $k$ for a system of length $N$ appear in \cite[Formulas (47) and (48)]{derrida1992exact} in the case $q=0$, $\alpha=\beta=1$, $\gamma=\delta=0$.
The first of them is generalized to arbitrary $\alpha,\beta$ in \cite[expressions (39) and (43)]{derrida1993exact} and in
\cite{schutz1993phase}. In particular, %
\citet[formula (3.3)]{schutz1993phase}
point out a factorization of the
expression for the difference $ \<\tau_{k}\>_N-\<\tau_{k+1}\>_N$ and
 discuss several implications of the factorization.

 Our goal in this section  is to explore Theorem \ref{T-rep} to gain quick access to some of such explicit formulas.

\subsection{Integral formulas for $K_N$}
Of course, \eqref{recalc-K2U} is an integral formula for $K_N$, expressing $K_N$
 as a constant times the integral of $(2+\sqrt{1-q}x)^N$ with respect to the law of $Z_1$.
  This is essentially   \cite[expression (6.1)]{uchiyama2004asymmetric} or \cite[expression (3.12)]{uchiyama2005correlation}.
 When $-1<A,B,C,D<1$,
 the law is absolutely continuous with the Askey-Wilson density \eqref{f-density}, which in general involves infinite products. For $q=0$, the infinite products vanish and the density takes a more  concise form. We get
\begin{equation}
  \label{Derrida+}
  K_N= M\int_{-2}^2 \frac{(2+x)^N \sqrt{4-x^2}}{(1+A^2-Ax)(1+B^2-Bx)(1+C^2-Cx)(1+D^2-Dx)}  dx\,,
\end{equation}
where
$$
M=\frac{(1-AB)(1-AC)(1-AD)(1-BC)(1-BD)(1-CD)}{2\pi(1-ABCD)}.$$
This formula simplifies further when  $\gamma=\delta=0$  and $\alpha,\beta>1/2$. From \eqref{ABCD} we get
$A=(1-\beta)/\beta$, $B=0$,  $C=(1-\alpha)/\alpha$, $D=0$; the values of $A,B$ need to be swapped when $\beta>1$
and similarly the values of $C,D$ need to be swapped when $\alpha>1$. So denoting by  $a=(1-\alpha)/\alpha$, $b=(1-\beta)/\beta$ the two
 non-zero values among $A,B,C,D$, expression
 \eqref{Derrida+} becomes
 \begin{equation}
  \label{K-int}
  K_N=\frac{\alpha+\beta-1}{2\pi \alpha\beta}\int_{-2}^2 \frac{(2+x)^N \sqrt{4-x^2}}{(1+a^2-ax) (1+b^2-bx) }  dx.
\end{equation}
With $x=2\cos\theta$  this is \cite[expression (B10)]{derrida1993exact}.
  (Another explicit formula is \cite[formula (42)]{blythe2000exact} for general $q$; \eqref{K-int} is $q=0$ of that formula.)

\subsection{Density profile}\label{Sect:DP} In this section we derive some additional explicit integral formulas.

In order to be able to use some known (or not so known) results on quadratic harnesses, we take $\gamma=\delta=0$.
 In this case \cite[formula (2.28)]{Bryc-Wesolowski-08} simplifies, and we
  can replace Markov process $(Z_t)$ by a better studied quadratic harness $(X_t)\sim QH(\theta,\eta,0,0,q)$, which in this setting becomes a
  ``bi-Poisson'' process from \cite{Bryc-Matysiak-Wesolowski-04b} (see also \cite[Section 5.2]{Bryc-Wesolowski-08}). This is a Markov process with three parameters $\eta,\theta\in\RR$,
  and $-1<q<1$; for $q\geq 0$ the parameters must satisfy additional constraint $1+\eta\theta>q$
 (see also \cite[Section 5.2]{Bryc-Wesolowski-08}).
  The parameters $\eta,\theta$ that are used in  \cite{Bryc-Matysiak-Wesolowski-04b} can be recalculated from \eqref{ABCD},
 using   formulas   %
 (\ref{eq:theta}--\ref{eq:gamma}).
 Since $AB=0$ and $CD=0$, these formulas specialize in our case to $\tau=\sigma=0$, $\gamma=q$ and
  $$\eta=\frac{ -(A+B)   \sqrt{1-q}}{\sqrt{(1-A C) (1-B C) (1-A D) (1-B D)  }}\label{eta},\;
  $$
  $$
  \theta=\frac{ -(D+C)  \sqrt{1-q}}{\sqrt{(1-A C) (1-B C) (1-A D) (1-B D)  }}\label{theta}\,.$$
Depending on the
 signs of $1-\beta-q$ and $1-\alpha-q$ in \eqref{ABCD},   there are four different cases  which all give the same final answer
 \begin{equation}
   \label{eta-theta}
\eta=\frac{\beta +q - 1}{\beta\vartheta } ,
   \;
   \theta=\frac{\alpha+q-1}{\alpha \vartheta},
 \end{equation}
 where $\vartheta$ is given by
\begin{equation}
  \label{chi}
\vartheta=\frac{\sqrt{\alpha+\beta+q-1}}{\sqrt{\alpha\beta}}.
\end{equation}
 We note that condition  $AC<1$  implies that $\alpha+\beta>1-q$, so $\vartheta>0$ is well defined, and  the constraint $1+\eta\theta>q$ mentioned above is fulfilled.

 From \cite[(2.28)]{Bryc-Wesolowski-08} we recalculate the process that appears in \eqref{QHsolution} as follows
 \begin{equation}
   \label{X2Z}
     1+t+\sqrt{1-q}Z_t=(1-q)\left(\vartheta X_t+\frac{t}{\beta}+\frac{1}{\alpha}\right).
 \end{equation}
 Since this linear function of $X_t$ will appear several times, we  denote
 \begin{equation}
   \label{L(x)}
L_t(x)=\vartheta x+\frac{t}{\beta}+\frac{1}{\alpha},
 \end{equation}
  and we put  $L(x)=L_1(x)$.
 The relevant version of \eqref{QHsolution} is now
 \begin{equation}
   \label{QH-biP}
   \left\langle \prod_{j=1}^N t_j^{\tau_j}\right\rangle_N=\frac{\EE\left(\prod_{j=1}^NL_{t_j}(X_{t_j})\right)}{\EE\left(L^N(X_1)\right)},
 \end{equation}
 and the normalizing constant \eqref{KN} is $K_N=\EE\left(L^N(X_1)\right)$.

  Denote by $P_{s,t}(x,dy)$ the transition probabilities of the Markov process $(X_t)$ and by $\pi_t(dy)=P_{0,t}(0,dy)$
  the univariate laws for the process started at $X_0=0$. (These measures depend on the parameters   $\eta,\theta$ given by \eqref{eta-theta},
  have absolutely continuous component, and may also have atoms.)

  Formula \eqref{QH-biP} can be restated as the quotient of  multiple integrals with respect to these measures. That is,
  $$\EE\left(L^N(X_1)\right)=\int L^N(x)\pi_1(dx)$$
  and
\begin{multline*}
\EE\left(\prod_{j=1}^NL_{t_j}(X_{t_j})\right)
\\
=\iint{\footnotesize \dots}\int L_{t_1}(x_1){\footnotesize \dots} L_{t_N}(x_N)P_{t_{N-1},t_N}(x_{N-1},dx_N){\footnotesize \dots} P_{t_{1},t_2}(x_{1},dx_2) \pi_{t_1}(dx_1).
\end{multline*}

  We will need the following result.
  \begin{theorem}
    \label{Thm-gen-biPoisson} %
    On polynomials $f,$ the infinitesimal generator
    $$A_t(f)(x):=\lim_{u\downarrow t} \frac{1}{u-t}\EE(f(X_u)-f(x)|X_t=x)$$ of the Markov process $(X_t)\sim QH(\theta,\eta,0,0,q)$ is given by
    \begin{equation}
      \label{A}
      A_t(f)(x)=(1+\eta x)\int \frac{\partial}{\partial x}\left(\frac{f(y)-f(x)}{y-x}\right) P_{q^2 t, t}(q(x-t\eta)+\theta,dy).
    \end{equation}

  \end{theorem}
More generally, if $f(x,t)$ is a polynomial in two variables $x,t$, then
\begin{equation}\label{BetterA}
  \lim_{u\downarrow t} \frac{1}{u-t}\EE\big(f(X_u,u)-f(x,t)\big|X_t=x\big)=\frac{\partial}{\partial t}f(x,t)+A_t(f(\cdot, t))(x).
\end{equation}
(In order not to interrupt the flow of this section, we postpone  the proof  of Theorem \ref{Thm-gen-biPoisson} to Appendix \ref{S:PofThm}.)

  We use Theorem \ref{Thm-gen-biPoisson} to explain the origin of  some of the ``miraculous'' explicit formulas for the average occupancy
  $\langle\tau_k\rangle_N$. Recall notation \eqref{chi} and  \eqref{eta-theta}.

  \begin{theorem}\label{T-miracle}
    Suppose   $\alpha+\beta>1-q$. %
     Then for $1\leq j\leq N-1$ we have %
  \begin{multline}
  \label{tau-ans}
      \langle \tau_j\rangle_N-\langle \tau_{j+1}\rangle_N
      \\
      =
      \vartheta^2\frac{\int  \left(\int L^{N-j-1}(y)P_{q^2, 1}(\theta+q(x-\eta),dy) \right)L^{j-1}(x)(1+\eta x)\pi_1(dx)}
      {\int L^N(x) \pi_1(dx) }.
    \end{multline}
  \end{theorem}
\begin{proof}
We begin with $\langle \tau_j\rangle_N$ with $1\leq j<N$.
Take $s<t<u$ so that we can apply Theorem \ref{T-rep}.
Using notation \eqref{L(x)}, we
need to identify the coefficient at $t$ in the expression
$$\EE\left(L_s^{j-1}(X_s)L_t(X_t)L_u^{N-j}(X_u)\right) $$
and then take the limit $u\to s$, and put $s=1$.
Since $\EE(X_t|X_s,X_u)=\frac{u-t}{u-s}X_s+\frac{t-s}{u-s}X_u$, compare \eqref{QH}, we see that the coefficient at $t$ is
\begin{equation}
  \label{tau-j1}
\EE\left(L_s^{j-1}(X_s)\frac{L_u^{N-j+1}(X_u)-L_s(X_s)L_u^{N-j}(X_u)}{u-s}\right).
\end{equation}
(This expression can also be obtained  directly from \eqref{MatrixSolution}.)

We will use Theorem \ref{Thm-gen-biPoisson} to compute the limit of \eqref{tau-j1} as $u\to s$. To do so, we write the limit as
\begin{multline}
  \lim_{u\to s} \EE\left(L_s^{j-1}(X_s)\frac{L_u^{N-j+1}(X_u) -L_s^{N-j+1}(X_s)}{u-s}\right)\\
  -  \lim_{u\to s}\EE\left(L_s^{j}(X_s)\frac{L_u^{N-j}(X_u)-L_s^{N-j}(X_s)}{u-s}\right)\\
  =  \EE\left(L_s^{j-1}(X_s)\lim_{u\to s}\EE\big(\frac{L_u^{N-j+1}(X_u)-L_s^{N-j+1}(X_s)}{u-s}\big|X_s\big)\right)  \\
  -  \EE\left(L_s^{j}(X_s)\lim_{u\to s}\EE\big(\frac{L_u^{N-j}(X_u)-L_s^{N-j}(X_s)}{u-s}\big|X_s\big)\right).
\end{multline}
Here we can pass to the limit under the integral  because %
 $X_s,X_u$ are bounded random variables, and %
 both conditional expectations
  are in fact polynomials in $s,u, X_s$;
the factor $(u-s)$ in the denominator cancels out by algebra. Compare
\cite[Proof of Prop. 2.3]{Bryc-Wesolowski-2013-evo}.

Using \eqref{BetterA} and taking $s=1$ we get
\begin{multline}
\label{tauj2}
  K_N\langle \tau_j\rangle_N =\frac{1}{\beta}\EE \left(L^{N-1}(X_1)\right)
  \\
   +\EE\left(L^{j-1}(X_1)\left[A_1(L^{N-j+1})(x)-L(x)A_1(L^{N-j})(x)\right]_{x=X_1}\right).
\end{multline}
We note that the first term on the right hand side of \eqref{tauj2} is
\begin{equation}
  \label{tau_N}
   \frac{1}{\beta}\EE \left(L^{N-1}(X_1)\right)=K_N\<\tau_N\>_N.
\end{equation}
Indeed,  it is easy to check from \eqref{MatrixSolution} (or from \eqref{QHsolution}) that (somewhat more generally) \comment{Uncommented next formula!!!}
\begin{equation}
  \label{tau_N+}
   \langle \tau_N\rangle_N=\frac{\delta}{\beta+\delta}+\frac{1}{\beta+\delta}\frac{K_{N-1}}{K_N}.
\end{equation}

\arxiv{\begin{proof}[Proof of \eqref{tau_N+}] From \eqref{QHsolution} we see that $\langle \tau_N\rangle_N$ is the coefficient at $t>1$ in the expression
$$
\EE\left((2+\sqrt{1-q}Z_1)^{N-1}(1+t+\sqrt{1-q}Z_t)\right)/\EE\left((2+\sqrt{1-q}Z_1)^N\right)
$$
From \cite[page 1233]{Bryc-Wesolowski-08} we read out that $$\left(\frac{AB\sqrt{1-q}Z_t-(A+B)}{1-ABt}\right)_{t>0}$$
is a martingale, so a coefficient at $t$ in $\EE((1+t+\sqrt{1-q}Z_t)|Z_1)$ is
\begin{multline*}
-\frac{AB}{1-AB}\sqrt{1-q}Z_1+1+\frac{A+B}{1-AB} \\ =
-\frac{AB}{1-AB}(2+\sqrt{1-q}Z_1)+\frac{(A+1) (B+1)}{1-A B}=\frac{\delta}{\beta+\delta}(2+\sqrt{1-q}Z_1)+\frac{1-q}{\beta+\delta}. \end{multline*}
Therefore,
$$\langle \tau_N\rangle_N =\frac{\delta}{\beta+\delta}+\frac{1-q}{\beta+\delta}\frac{\EE(2+\sqrt{1-q}Z_1)^{N-1}}{\EE(2+\sqrt{1-q}Z_1)^{N}}.$$
By \eqref{recalc-K2U}, this ends the proof.
\end{proof}
}
We now simplify the expression in the square brackets on the right hand side of \eqref{tauj2}. To do so, we  combine together  two expressions that appear  under the integral in  formula \eqref{A} for $A_1(\cdot)$. These expressions are
\begin{multline*}
\frac{\partial}{\partial x}\frac{L^{N-j+1}(y)-L^{N-j+1}(x)}{y-x}-L(x)\frac{\partial}{\partial x}\frac{L^{N-j}(y)-L^{N-j}(x)}{y-x}
\\
=\vartheta
\left(\frac{\partial}{\partial x}\frac{L^{N-j+1}(y)-L^{N-j+1}(x)}{L(y)-L(x)}-
L(x)\frac{\partial}{\partial x}\frac{L^{N-j}(y)-L^{N-j}(x)}{L(y)-L(x)}\right).
\end{multline*}
After a linear change of variables $(x,y)$ to $(\ell_x,\ell_y)=(L(x),L(y))$, this becomes
\begin{multline*}
\vartheta^2\left(\frac{\partial}{\partial \ell_x}\frac{\ell_y^{N-j+1}-\ell_x^{N-j+1}}{\ell_y-\ell_x}-
\ell_x\frac{\partial}{\partial \ell_x}\frac{\ell_y^{N-j}-\ell_x^{N-j}}{\ell_y-\ell_x}\right) \\
=\vartheta^2
\frac{\ell_y^{N-j}-\ell_x^{N-j}}{\ell_y-\ell_x}= \vartheta^2
\frac{L^{N-j}(y)-L^{N-j}(x)}{L(y)-L(x)}
.
\end{multline*}
Together with \eqref{tau_N}  and \eqref{A}, this gives
\begin{multline}
  \label{tauj3}
   \langle \tau_j\rangle_N= \langle \tau_N\rangle_N \\ +
   \tfrac{\vartheta^2}{K_N} \int L^{j-1}(x)\int  \tfrac{L^{N-j}(y)-L^{N-j}(x)}{L(y)-L(x)}
  P_{q^2,1}(\theta+q(x-\eta),dy) (1+\eta x)\pi_1(dx).
\end{multline}
Of course, the fraction in the integrand on the right hand side of \eqref{tauj3}
can be written as  $\sum_{k=1}^{N-j}L^{k-1}(x)L^{N-j-k}(y)$.
So taking the difference $\langle \tau_j\rangle_N -\langle \tau_{j+1}\rangle_N$  of two consecutive expressions given by \eqref{tauj3}, after cancelations,
the only remaining term corresponds to $k=1$ and we get
 \eqref{tau-ans}.
\end{proof}
\begin{remark}
  \cite[Section 7.1]{uchiyama2004asymmetric} states formulas for the correlation functions of any order in terms of limits of integrals that in principle imply, and extend, our formula \eqref{tau-ans}.
  Our contribution here lies in
  ``explicit formula for the limit'',  which allows us to express  the final answer  as an integral.
\end{remark}

From \eqref{tau-ans} we see that in the case when $q=0$ the expression for $\langle\tau_j-\tau_{j+1}\rangle_N$ factors into a product of two integrals.
Now we will investigate the factors in more
detail.

The infinitesimal generator \eqref{A} for the Markov process in this case can be made
 more explicit by specializing  \cite{Bryc-Wesolowski-2013-evo}; to get the bi-Poisson process with $q=0$,
we take  $\tau=\sigma=0$, and use \eqref{eta-theta} for the other two parameters.
To avoid  atoms, we  restrict the range of  parameters $\alpha,\beta$.

 \begin{theorem}\label{Thm-tau:q=0}
   If $q=0$ and $\alpha,\beta>1/2$ then we have  factorization %
    \begin{equation}\label{Schutz-like}
      \langle \tau_j\rangle_N-\langle \tau_{j+1}\rangle_N
 =
      \frac{\alpha+\beta-1}{\alpha\beta} \frac{F_{j}(\beta)F_{N-j}(\alpha)}{\hat K_N(\alpha,\beta)},
    \end{equation}
    where
    \begin{equation}
      \label{One-Factor-a}
      F_j(\alpha)=  \int_{-2}^2 \frac{(2+z)^{j-1} \sqrt{4-z^2}}{  (1+a^2-az) }\,  dz
    \end{equation}
    depends only on $\alpha$ through $a=(1-\alpha)/\alpha$,
      \begin{equation}
      \label{One-Factor-b}
      F_{N-j}(\beta)=  \int_{-2}^2 \frac{(2+z)^{N-j-1} \sqrt{4-z^2}}{(1+b^2-bz)  }\,  dz
    \end{equation}
    depends only on $\beta$ through $b=(1-\beta)/\beta$, and the normalizing constant
      \begin{equation}
      \label{One-Factor}
      \hat K_N=  \int_{-2}^2 \frac{(2+z)^N \sqrt{4-z^2}}{(1+a^2-az) (1+b^2-bz) }\,  dz
    \end{equation}
    is   \eqref{K-int}, up to  a multiplicative factor.
 \end{theorem}
 This is an integral form  of \cite[expression (3.3)]{schutz1993phase} (which does not restrict  ranges of the parameters). Our proof establishes factorization
 \eqref{Schutz-like} for all positive $\alpha, \beta$ such that  $\alpha+\beta>1$,
  but the integral expressions \eqref{One-Factor-a}, \eqref{One-Factor-b}, and \eqref{One-Factor} when $\alpha<1/2$ or $\beta<1/2$ would have to be modified to include also an atomic part.

 \begin{proof}
 Taking $q=0$ in \eqref{tau-ans}, we see that  the expression for $\langle \tau_k\rangle_N-\langle \tau_{k+1}\rangle_N$ involves the product of two
 integrals with respect to $P_{0,1}(\theta,dy)$ and $\pi_1(dx)$.
   The first step is to read out from \cite[Remark 4.2]{Bryc-Wesolowski-2013-evo} that
 for $q=0$ we have $P_{0,1}(\theta,dy)=(1+\theta x)\pi_1(dy)$.

   Recall that $\pi_1$ has no atoms when $A^2,C^2<1$, compare \eqref{u(t)}.
   An equivalent condition: $\theta^2<1+\eta\theta$ and $\eta^2<1+\eta\theta$ can be read out from \cite[Section 3]{Bryc-Wesolowski-04}. Both forms of the condition are equivalent to $(1-\alpha)^2<\alpha^2$ and $(1-\beta)^2<\beta^2$, see \eqref{eta-theta} or \eqref{ABCD}. These inequalities hold for $\alpha,\beta>1/2$, so this is the case when measure $\pi_1$ has no atoms.

    Since we   have no atoms,
   from \cite{Bryc-Wesolowski-04} we read out  that
 \begin{equation}\label{EQ: distr}
\pi_{1}(dx)=
\frac{
\sqrt{4  (1+\eta\theta)-(x-\eta-\theta)^2}}
{2\pi(1+\eta x )(1+\theta x)}
1_{(x-\eta-\theta)^2<4(1+\eta\theta)}.
\end{equation}
So the two integrals in  \eqref{tau-ans} take similar forms:
$$\int_{\eta+\theta-2\sqrt{1+\eta\theta}}^{\eta+\theta+2\sqrt{1+\eta\theta}}  L^{j-1}(x)\frac{
\sqrt{4  (1+\eta\theta)-(x-\eta-\theta)^2}}
{2\pi (1+\theta x)}
\,dx$$
and
$$\int_{\eta+\theta-2\sqrt{1+\eta\theta}}^{\eta+\theta+2\sqrt{1+\eta\theta}}  L^{N-j-1}(x)\frac{
\sqrt{4  (1+\eta\theta)-(x-\eta-\theta)^2}}
{2\pi (1+\eta x)}
\,dx.$$
We now substitute $z=(x-\eta-\theta)/\sqrt{1+\eta \theta}$, and revert back to the parameters $\alpha,\beta$, see \eqref{eta-theta} and then back to $A,C$, see
\eqref{ABCD}. The substitution is really the backwards conversion from the process $(X_t)$ to $(Z_t)$, see \eqref{X2Z}, so the linear form
 $L(x)$ turns into the parameter-free expression $2+z$, converting the two integrals into \eqref{One-Factor-a} and \eqref{One-Factor-b}.

 \end{proof}

In particular, with $q=0$ and $\alpha=\beta=1$, we have $\eta=\theta=0$ and
$$
\pi_1(dx)=\frac{\sqrt{4-x^2}}{2\pi}1_{|x|<2}\, dx $$
 is the Wigner semicircle law, whose even moments   are the Catalan numbers. In this case, substituting $y^2=2+x$ we get  %
$$\int_{-2}^2 (2+x)^n  \frac{\sqrt{4-x^2}}{2\pi} dx =\int_{-2}^2 y^{2n+2}\frac{\sqrt{4-y^2}}{2\pi} dy= C_{n+1},
$$
 where 
 $C_{n}=\frac{1}{n+1}\left(^{2n}_n\right)$ is the $n$-th Catalan number.
 This shows that
 $$ \langle \tau_{N-k}\rangle_N-\langle \tau_{N-k+1}\rangle_N=\frac{C_{N-k}C_k}{C_{N+1}},$$
  which is essentially \cite[formula (47)]{derrida1992exact}, see also \cite[formula (84)]{derrida1993exact}.
   (It is worth pointing out  that \cite[formula (48)]{derrida1992exact} encodes an explicit closed-form expression for
  the ``incomplete convolution'' $\sum_{j=0}^k  C_{N-j}C_j$.)
\section{Limits of finite  ASEPs}\label{Sect:semininf}

 \citet[Theorem 3.10]{Liggett-1975} proves that for fixed values of $\alpha,\beta,\gamma,\delta,q$ the
 sequence %
 of the invariant measures for the  ASEPs on $\{1,\dots,N\}$ converges weakly as $N\to\infty$ to a probability
 measure   $\mu$ on $\{0,1\}^{\NN}$.
  This measure  is a (non-unique) stationary distribution of an ASEP on the semi-infinite lattice $\NN$,
  and  arises as a
 limiting measure of an infinite system
started from an appropriate  product measure, see \cite[Theorem 1.8]{Liggett-1975}.
 \citet[Theorem 3.2]{grosskinsky2004phase}   introduced  matrix representation
 for $\mu$ in  the totally asymmetric process ($q=0$);
 \citet[Theorem 1.1]{sasamoto2012combinatorics} extended matrix
 representation to general $q$.  Ref. \cite{duhart2014semi}, see also \cite{gonzalez2015large},
 use the matrix representation to prove the large deviations principle for the ASEP on a semi-finite lattice when $q=0$.

 Here we show how to extend Theorem \ref{T-rep} to represent limits of ASEP on $\{1,\dots,N\}$ as $N\to\infty$, and we deduce
 a version of \cite[Theorem 3.10]{Liggett-1975}.
  Note however that as in  Theorem \ref{T-rep} we do not cover the entire
   range of the
  admissible parameters  $\alpha,\beta,\gamma,\delta,q$ due to the restriction $AC<1$.

We first prove  a somewhat more general limit result which depends on additional parameter $u\geq 1$.
\begin{theorem}
\label{Thm-semin-u}   Suppose that the parameters of ASEP are such that $AC<1$   with
$A,B,C,D$ defined by \eqref{ABCD}. Then for fixed $u\geq 1$ and every $K\geq 1$, there exists
  a unique probability measure $\mu_{K,u}$ on the
 subsets of $\{0,1\}^K$ such that
\begin{equation}\label{lim-u}
\int_{\{0,1\}^K}\prod_{j=1}^K t_j^{\tau j} d\mu_{K,u}=\lim_{N\to\infty}\frac{\big\langle \prod_{j=1}^K t_j^{\tau j} u^{\tau_{K+1}+\dots\tau_N}\big\rangle_N}{\langle u^{\tau_{K+1}+\dots\tau_N}\rangle_N}.
\end{equation}
Furthermore, for $t_1\leq \dots t_k\leq u$ we have
\begin{equation}
  \label{u-Z}
\int_{\{0,1\}^K}\prod_{j=1}^K t_j^{\tau j} d\mu_{K,u}= \frac{\EE\left(\prod_{j=1}^K(1+t_j+\sqrt{1-q}\,\widetilde{Z}_{t_j})\right)}{\EE\left((2+\sqrt{1-q}\,\widetilde{Z}_1)^K\right)},
\end{equation}
where  $(\widetilde {Z}_t)_{0\leq t\leq u}$ is  the Markov process from \eqref{JW-Z} with
parameters
\begin{equation}
  \label{tilde-A-u}
  \widetilde A_u =\begin{cases}
    \frac{C}{u} &\mbox{if $u<C^2$},\\
    u^{-1/2} & \mbox{if $C^2\leq u\leq 1/A^2$},\\
    A & \mbox{ if $u>1/A^2$},
  \end{cases}
\end{equation}
\begin{equation}
  \label{tilde-B-u}
  \widetilde B_u =\begin{cases}
    \frac{1}{C} &\mbox{if $u<C^2$},\\
    u^{-1/2} & \mbox{if $C^2\leq u\leq 1/A^2$},\\
    \frac{1}{Au} & \mbox{ if $u>1/A^2$},
  \end{cases}
\end{equation}
and $\widetilde{C}_u=C,\widetilde{D}_u=D$ in place of $A,B,C,D$, respectively.
\end{theorem}
 The proof is somewhat involved and appears in Section \ref{Sec:proof-semi-u}.

As mentioned above, it turns out that
$\{\mu_{K,1}:K\in\NN\}$ is a consistent family of   measures which are the marginals of a (unique) probability measure $\mu$ on  the Borel subsets of $\{0,1\}^\NN$ ;
this is the case analyzed in \cite[Theorem 3.10]{Liggett-1975}.
\citet{grosskinsky2004phase} and \citet{sasamoto2012combinatorics} give matrix representations for $\mu$.
\begin{theorem}\label{T-semiinfinite}
Suppose that the parameters of ASEP are such that $AC<1$, where $A,B,C,D$ are defined in
\eqref{ABCD}. Then there exists a probability measure $\mu$ on the Borel subsets of $\{0,1\}^\NN$ such that  for
$K\geq 1$ and $0\leq t_1\leq t_2\leq \dots \leq t_K \leq 1$ we have

\begin{equation}\label{Lim-tau}
 \int_{\{0,1\}^{\NN}} \prod_{j=1}^K t_j^{\tau_j}  d\mu  = \lim_{N\to\infty}\left\langle \prod_{j=1}^K t_j^{\tau_j}\right\rangle_N =
 \frac{\EE\left(\prod_{j=1}^K(1+t_j+\sqrt{1-q}\,\widetilde{Z}_{t_j})\right)}{\EE\left((2+\sqrt{1-q}\,\widetilde{Z}_1)^K\right)},
\end{equation}
where $(\widetilde Z_t)_{0\leq t\leq 1}$ is the Markov process from Theorem \ref{Thm-semin-u} for $u=1$.
\end{theorem}
\begin{proof}
 The proof of Theorem \ref{Thm-semin-u}
   identifies the transition probabilities of the Markov
  process $(\widetilde Z_t)$ and notes that $\widetilde Z_u$ is non-random.
  When $u=1$ this means that $(2+\sqrt{1-q}\,\widetilde{Z}_1)$ is a non-random constant,
  and hence the probability generating functions on the left hand side of \eqref{u-Z} define a consistent family of probability measures $\{\mu_{K,1}:K\geq 1\}$.
  That is, the probability measure $\mu_{K,1}$   is the marginal of $\mu_{K+1,1}$ with $\{0,1\}^{K+1}=\{0,1\}^{K}\times \{0,1\}$.
 To see this,  write \eqref{u-Z} for $K+1$ and take $t_1<t_2<\dots<t_K<t_{K+1}=1$.
 On the right hand side of \eqref{u-Z} we get
\begin{multline*}
   \frac{\EE\left(\prod_{j=1}^K(1+t_j+\sqrt{1-q}\,\widetilde{Z}_{t_j}) (2+\sqrt{1-q}\,\widetilde{Z}_1)\right)}{(2+\sqrt{1-q}\,\widetilde{Z}_1)^{K+1}}
 \\=\frac{\EE\left(\prod_{j=1}^K(1+t_j+\sqrt{1-q}\,\widetilde{Z}_{t_j})\right)}{(2+\sqrt{1-q}\,\widetilde{Z}_1)^K}
\end{multline*}because $(2+\sqrt{1-q}\,\widetilde{Z}_1)$ is a non-random factor.

\end{proof}

We remark that parameter $\widetilde A_1$
can be written using function  \eqref{kappa}. We have $\widetilde{A}_1=\kappa_+(J,-J)$ with %
$$J=\frac{1-q}{\|2+\sqrt{1-q}Z_1\|_\infty},$$
 where $\|2+\sqrt{1-q}Z_1\|_\infty=2+U(1)$ is read out from \eqref{u(t)}. (In physics literature, $J$ is interpreted as ``current".) Somewhat more generally,
 $\widetilde{A}_u=\kappa_+(\tilde \beta,\tilde\delta)$ with ``non-physical" $\widetilde \delta=-(1-q)/\|1+u+\sqrt{1-q}Z_u\|_\infty<0$ and
   $\widetilde \beta=- u\widetilde \delta>0$.

\arxiv{
Using \eqref{tilde-A-u} and \eqref{tilde-B-u}, from \eqref{Jacek} we  read out that $\widetilde \delta=-\widetilde A\widetilde B \widetilde \beta=-\widetilde \beta/u$ and
$$
\widetilde \beta= \frac{1-q}{1+\widetilde A_u+\widetilde B_u+\widetilde A_u\widetilde B_u}=
\frac{1-q}{1+\widetilde A_u+1/(u\widetilde A_u)+1/u}
$$
This gives
$$ \tilde \beta=   \begin{cases}
\frac{(1-q)Cu}{(1+C)(C+u)} & u<C^2\\
\frac{(1-q)u}{ (1+\sqrt{u} )^2 } & C^2\leq u\leq 1/A^2\\
\frac{(1-q)Au}{(1+A)(1+Au)} & u>1/A^2
\end{cases}
$$
That is,
$\tilde \beta=(1-q)/ \zeta=(1-q)u /\|1+u+\sqrt{1-q}Z_u\|_\infty$, see \eqref{zeta}.
}

 Since  the right hand sides of formulas \eqref{Lim-tau} and \eqref{QHsolution} have the same form, we recover an observation from \cite{sasamoto2012combinatorics}, that the finite
correlation functions involving the leftmost $K$ sites of the
stationary measures of the semi-infinite ASEP can be obtained from the stationary
distribution for the finite ASEP on a lattice of $K$  sites
 with non-physical parameters
$\beta=J$ and $\delta=-J$ used in \eqref{ABCD}. However,  there are also significant differences between these two
representations: while $(Z_t)$ is well defined for all $t>0$, $(\widetilde Z_t)$ is defined only for $0\leq t\leq 1$, and while
$Z_1$ is random, $\widetilde Z_1=\|Z_1\|_\infty$ is deterministic, so the expected value in the denominator on the right hand
side  of \eqref{Lim-tau} can be omitted.

\begin{remark}
  \label{Rem-Prod} The proof of Theorem \ref{Thm-semin-u}
   identifies the transition probabilities of the Markov
  process $(\widetilde Z_t)$ as the Askey-Wilson measures.
  However, when $C^2\geq u$,
 the general theory is somewhat difficult to apply as
  we have $\widetilde A_u \widetilde C_u\geq 1$ and $\widetilde B_u \widetilde C_u=1$.
In this case, the process $(\widetilde Z_t)$ turns out to be non-random;
this  can  be seen from the formula for the variance
\cite[(2.15)]{Bryc-Wesolowski-08} which is zero when the product of the parameters is $1$.
 (This is similar to
 the case $N=0$ in \cite[Section 4]{Bryc-Wesolowski-08} where all the parameters were
 assumed positive while here we have $\widetilde D_u<0$.)
So for $C^2\geq u$, we have $\widetilde Z_t=\EE(\widetilde Z_t)$ and
from the formula for the mean in \cite[(2.14)]{Bryc-Wesolowski-08} we recalculate that
$
\sqrt{1-q}\widetilde Z_t= t/C+C$.
 \arxiv{
 Here are the details of this calculation. By \cite[(2.14)]{Bryc-Wesolowski-08}
 \begin{multline*}
    \sqrt{1-q}\widetilde Z_t=\frac{[\widetilde A+\widetilde B-\widetilde A\widetilde B(C+D)]t +C+D-CD(\widetilde A+\widetilde B) }{1-\widetilde A\widetilde B C D}
 \\
 =
 \frac{[C/u+1/C-(C+D)/u]t +C+D-CD(C/u+1/C) }{1-C D/u}
 \\
 =\frac{[ 1/C- D/u]t +C -C^2D/u }{1-C D/u} = t/C+C,
 \end{multline*}
 noting that $1-C D/u>0$ as  $CD<0$ by assumption.
 }
Thus in this case formula \eqref{u-Z}  shows that $\mu_{K,u}$ is a product of Bernoulli measures with
 $\mu_{K,u}(\tau_j=1)$ determined
 as the coefficient at $t$ in the expression
$$
\frac{1+t+\sqrt{1-q}\widetilde Z_t}{1+u+\sqrt{1-q}\widetilde Z_u}=\frac{1+t+t/C+C}{1+u+u/C+C}.
$$
A calculation gives $\mu_{K,u}(\tau=1)=1/(C+u)$.
 When $u=1$ we get  $\mu(\tau_k=1)=\frac{1}{1+C}=\rho_0$, which is in agreement with  \cite[Theorem 1.8(b)]{Liggett-1975}: from \eqref{rho} we see that $C>1$ corresponds to
$\rho_1<\rho_0<1/2$, where the first inequality is by assumption $AC<1$.
In the notation of  \cite{Liggett-1975}
with $\rho=1-\rho_1$ and $\la=\rho_0$ this is the case $1-\rho<\la<1/2$, so by \cite[Theorem 1.8]{Liggett-1975},
as time goes to infinity, the law of the semi-finite ASEP started from a product Bernoulli measure with average density $1-\rho_1$  converges
to  the product Bernoulli measure with  average density $\rho_0$.

\end{remark}

\begin{remark}
   One can check that measures $\mu_{K,u}$ are consistent only in the two already considered cases: when $u=1$ or when $u\leq C^2$.

\end{remark}
To verify this claim, let us equate the first  marginal of $\mu_{2,u}$ with $\mu_{1,u}$.
  Consistency implies that for $t\leq 1$ we have
  \begin{multline}
    \label{consistency}
    \EE\left( 2+\sqrt{1-q}\widetilde Z_1\right) \EE\left((1+t+\sqrt{1-q}\widetilde Z_t)(2+\sqrt{1-q}\widetilde Z_1)\right)
  \\
  =\EE\left((1+t+\sqrt{1-q}\widetilde Z_t)\right)\EE\left((2+\sqrt{1-q}\widetilde Z_1)^2\right).
  \end{multline}
  All four expectations can be computed from  \cite[ formulas (2.14) and (2.16)]{Bryc-Wesolowski-08}.
The linear terms are determined from
  \begin{equation}
    \EE\left(1+t+\sqrt{1-q}\widetilde Z_t\right) =1+t+\frac{C+D}{u-CD}(u-t)+(\widetilde A u +1/\widetilde A)\frac{t-CD}{u-CD}.
  \end{equation}
 \arxiv{Indeed, from \cite[(2.14)]{Bryc-Wesolowski-08} we have \begin{multline*}
    \EE\left(1+t+\sqrt{1-q}\widetilde Z_t\right) =
    1+t+\frac{[\widetilde A+\widetilde B-\widetilde A\widetilde B(C+D)]t +C+D-CD(\widetilde A+\widetilde B) }{1-\widetilde A\widetilde B C D}
 \\
 =1+t+\frac{[\widetilde A+1/(\widetilde A u)- (C+D)/u]t +C+D-CD(\widetilde A+1/(\widetilde A u)) }{1-  C D/u}
 \\=
 1+t+\frac{[\widetilde A u+1/\widetilde A - (C+D)]t +(C+D)u-CD(\widetilde A u+1/\widetilde A ) }{u-  C D}
\\ =1+t+\frac{C+D}{u-CD}(u-t)+(\widetilde A u +1/\widetilde A)\frac{t-CD}{u-CD}
  \end{multline*}}

  To compute
  $\EE\left((1+t+\sqrt{1-q}\widetilde Z_t)(2+\sqrt{1-q}\widetilde Z_1)\right)$ we need the above formula, and the covariance.
  From \cite[  (2.16)]{Bryc-Wesolowski-08}, we recalculate
   \begin{equation}
  cov(\widetilde Z_t,\widetilde Z_1)
  =\frac{ (u-1) (\tilde A C-1) (\tilde A D-1) (\tilde A u-C) (\tilde A u-D) (C D-t)}{\tilde A^2 (u-C D)^2 (C D q-u)}.
  \end{equation}
  \arxiv{ Indeed, from \cite[(2.16)]{Bryc-Wesolowski-08} we have \begin{multline*}
  cov(\widetilde Z_t,\widetilde Z_1)=  \frac{ (1-\tilde A C)(1-\tilde B C)(1-\tilde A D)(1-\tilde B D)}{(1-\tilde A\tilde B C D)^2(1-q\tilde A\tilde B CD)}
  (t-CD)(1-\tilde A\tilde B)
  \\
  =\frac{ (1-\tilde A C)(1-\frac{C}{\tilde A u})(1-\tilde A D)(1-\frac{D}{\tilde A u})}{(1-  C D/u)^2(1-q  CD/u)}
  (t-CD)(1-1/u)\\
  =\frac{ (u-1) (\tilde A C-1) (\tilde A D-1) (\tilde A u-C) (\tilde A u-D) (C D-t)}{\tilde A^2 (u-C D)^2 (C D q-u)}
  \end{multline*}}

With the above formulas,  calculation converts the consistency condition \eqref{consistency}   into the following equation:
  $$
  \tfrac{(C+1) (D+1) (q-1) (t-1) (u-1) (\widetilde A C-1) (\widetilde A D-1) (C-\widetilde A u) (D-\widetilde A u)}{\widetilde A (C D-u) (C D q-u)}=0.
  $$
  Since $\widetilde A\geq 0$,  and $-1<D\leq 0$, possible solutions are:
  \begin{enumerate}
    \item $u=1$ (which   is consistent by Theorem \ref{T-semiinfinite});
    \item $\widetilde A= C/u$ which by \eqref{tilde-A-u} holds for $u< C^2$. For $C^2\leq u\leq 1/A^2$ equality $\widetilde A= C/u$  implies  $u=C^2$. As we remarked $u\leq C^2$ gives  a (consistent)
   family of product Bernoulli measures. For $u>1/A^2$ equality $\widetilde A= C/u$ cannot hold as it reduces to $A=C/u$. The latter implies $AC=uA^2>1$ contradicting the assumption $AC<1$.
    \comment{Added details for two cases when $u\geq C^2$.}
    \item $\widetilde A= 1/C$. Since we assume $AC<1$ and we have \eqref{tilde-A-u}, this is possible only when $u=C^2$, a case already covered;
 \item $D=-1$, which is not possible; we already noted that $\kappa_-(\alpha,\gamma)>-1$.
  \end{enumerate}

 \subsection{Proof of Theorem \ref{Thm-semin-u}}\label{Sec:proof-semi-u}
In view of \eqref{QHsolution}, we need only to show that
\begin{multline}
  \label{LimZ} \lim_{N\to\infty}\frac{\EE\left(\prod_{j=1}^K(1+t_j+\sqrt{1-q}\,Z_{t_j}) (1+u+\sqrt{1-q}\,Z_u)^{N-K}\right)}{\EE\left((2+\sqrt{1-q}\,Z_1)^K (1+u+\sqrt{1-q}\,Z_u)^{N-K}\right)}
\\
=\frac{\EE\left(\prod_{j=1}^K(1+t_j+\sqrt{1-q}\,\widetilde{Z}_{t_j})\right)}{\EE\left((2+\sqrt{1-q}\,\widetilde{Z}_1)^K\right)}.
\end{multline}
Throughout the proof, we will use auxiliary processes $(\hat Z_t)$ and $(\hat{\tilde{Z}}_t)$ which will be  the Markov process from \eqref{JW-Z} in Section \ref{AMP},
with parameters $\hat{A}, \hat{B}, \hat{C},\hat{D}$  or $\hat{\tilde{A}}, \hat{\tilde{B}}, \hat{\tilde{C}},\hat{\tilde{D}}$ used in place of parameters $A,B,C,D$.
(The ``hatted" processes are the time inversions
of the processes we want.) Since $u>1$ is fixed, in  notation we suppress the dependence of the parameters on $u$.

\subsubsection*{Step 1: Time inversion}
To facilitate the use of martingale polynomials, the first step is to re-write the left hand side of \eqref{LimZ} in
reverse order. From \eqref{Y-univariate} and \eqref{Y-transitions}
we see that the Markov process $\hat Y_t:=Y_{1/t}$ is also an Askey-Wilson process, with swapped pairs of parameters: $\hat A=C$, $\hat B=D$, $\hat C=A$, $\hat D=B$. Therefore
the time-inversion
\begin{equation}
  \label{Z-inv}
  \hat Z_t:=t Z_{1/t}, \quad t>0
\end{equation}
is also a process  from \eqref{JW-Z} in Section \ref{AMP}, with the above parameters $\hat A,\hat B, \hat C, \hat D$. In particular, it has martingale polynomials that arise from Askey-Wilson polynomials
(with new parameters) just as in \eqref{w2r}.

We now re-write \eqref{LimZ} in terms of $(\hat Z_t)$. Let $\hat u=1/u<1$ and
\begin{equation}
  \label{t2s}
  1/s_j=t_{K+1-j}
\end{equation} so that $\hat u\leq s_1\leq s_2\leq\dots\leq s_k$.
Then
\begin{multline}\label{run-out-of-names}
\prod_{j=1}^K(1+t_j+\sqrt{1-q}\,Z_{t_j})=t_1\dots t_K  \prod_{j=1}^K(1+\tfrac1{t_j}+\sqrt{1-q}\tfrac1{t_j}Z_{t_j})
\\  =t_1\dots t_K  \prod_{j=1}^K(1+\tfrac1{t_j}+\sqrt{1-q}\tfrac1{t_j}Z_{t_j})
=t_1\dots t_K  \prod_{j=1}^K(1+s_j+\sqrt{1-q}s_jZ_{1/s_j})
\\
=t_1\dots t_K  \prod_{j=1}^K(1+s_j+\sqrt{1-q}\hat Z_{s_j}).
\end{multline}
Therefore, to prove  \eqref{LimZ} we need to find a process $(\hat{\tilde Z}_t)$ such that for
 $\hat u\leq s_1\leq s_2\leq\dots\leq s_k$ we have
\begin{multline}
  \label{Lim-hatZ} \lim_{N\to\infty}\frac{\EE\left( (1+\hat u+\sqrt{1-q}\hat Z_{\hat u})^{N-K}\prod_{j=1}^K(1+s_j+\sqrt{1-q}\hat Z_{s_j})\right)}
  {\EE\left((1+\hat u+\sqrt{1-q}\hat Z_{\hat u})^{N-K}(2+\sqrt{1-q}\hat Z_1)^K\right)}
\\
 =\frac{\EE\left(\prod_{j=1}^K(1+s_j+\sqrt{1-q}\,\hat{\tilde{Z}}_{s_j})\right)}{\EE\left((2+\sqrt{1-q}\,\hat{\tilde{Z}}_1)^K\right)}.
\end{multline}

To prove \eqref{Lim-hatZ}, we will find a Markov process $(\hat{\tilde Z}_t)$ such that
\begin{multline}
  \label{Lim-hatZ+} \lim_{N\to\infty}\frac{\EE\left( (1+\hat u+\sqrt{1-q}\hat Z_{\hat u})^{N-K}\prod_{j=1}^K(1+s_j+\sqrt{1-q}\hat Z_{s_j})\right)}
  {\EE\left((1+\hat u+\sqrt{1-q}\hat Z_{\hat u})^N\right)}
\\
=\frac{\EE\left(\prod_{j=1}^K(1+s_j+\sqrt{1-q}\,\hat{\tilde{Z}}_{s_j})\right)}{\EE\left((1+\hat u+\sqrt{1-q}\,\hat{\tilde{Z}}_{\hat u})^K\right)}.
\end{multline}
This will end the proof, as we can write the left hand side of \eqref{Lim-hatZ} as the quotient of two such limits, one for
$s_1\leq s_2\leq\dots\leq s_k$ and the second one for
$s_1=s_2=\dots=s_K=1$.

\comment{Added clarification}
We have thus reduced the proof of \eqref{LimZ} to the proof of \eqref{Lim-hatZ+}.
The next step is to compute the limit on the right hand side of \eqref{Lim-hatZ+}.

\subsubsection*{Step 2: Calculating  the limit \eqref{Lim-hatZ+}}
Recall that process $(\hat Z_s)$ has martingale polynomials $\{\hat r_n(x;s)\}$. Any polynomial in variable $\hat Z_s$ can be written as a linear combination of the martingale polynomials $\{\hat r_n(x;s)\}$.
Therefore, conditional expectation of any polynomial in $\hat Z_s$  with respect to past $\sigma$-field generated by  $\{\hat Z_{v}:v\leq s'\}$  for $s'<s$, is
a polynomial of the same degree in the variable
 $\hat Z_{s'}$, as it is  the same linear combination of the polynomials  $\{\hat r_n(x,s')\}$.
Using this property recursively with $s=s_j$ and   $s'=s_{j-1}$ we see that $ \EE( \prod_{j=1}^K(1+s_j+\sqrt{1-q}\hat Z_{s_j})|\hat Z_{\hat u})$
is a polynomial of degree $K$ in $\hat Z_{\hat u}$. After a linear change of the variable, we see that
there exists a   polynomial $p_K$ of degree $K$ such that
\begin{equation}
  \label{Eq:def-p_K}
  \EE\large( \prod_{j=1}^K(1+s_j+\sqrt{1-q}\hat Z_{s_j})\large|\hat Z_{\hat u}\large)=p_K(1+\hat u+\sqrt{1-q}\hat Z_{\hat u}).
\end{equation}
Thus, with $Z=(1+\hat u+\sqrt{1-q}\hat Z_{\hat u})$ we can write the numerator on the left hand side of \eqref{Lim-hatZ+} as follows.
$$
  \EE\left( Z^{N-K}\prod_{j=1}^K(1+s_j+\sqrt{1-q}\hat Z_{s_j})\right)=
\EE\left(p_K(Z)Z^{N-K}\right). $$

This representation allows us to compute the limit on the left hand side of \eqref{Lim-hatZ+} as follows.
\begin{lemma}
  \label{Lem-int} If $Z\geq 0$ is a bounded non-zero random variable  and $p$ is a polynomial then  $$\lim_{n\to\infty}\frac{\EE(p(Z)Z^{n})}{\EE(Z^n)}=p(\|Z\|_\infty).$$
\end{lemma}
\begin{proof}%
By linearity, we compute the limit for the  monomials $p(Z)=Z^k$. Denote    $z:=\|Z\|_\infty >0$. We want to show that $\lim_{n\to\infty}\frac{\EE(Z^{n+k})}{\EE(Z^n)}=z^k$.
Since $Z^{n+k}\leq z^k Z^n$, $$\limsup_{n\to\infty}\frac{\EE(Z^{n+k})}{\EE(Z^n)}\leq z^k\limsup_{n\to\infty}\frac{\EE(Z^{n})}{\EE(Z^n)} =z^k.$$

On the other hand, for  $\eps\in(0,z)$ we have
\begin{multline*}
   \frac{\EE(Z^{n+k})}{\EE(Z^n)}\geq \frac{\EE(Z^{n+k}I_{Z>z-\eps})}{\EE(Z^n)}\geq \frac{(z-\eps)^k\EE(Z^{n}I_{Z>z-\eps})}{\EE(Z^{n})}
\\ =  (z-\eps)^k  -(z-\eps)^k \frac{\EE(Z^{n}I_{Z\leq z-\eps})}{\EE(Z^{n})}.
 \end{multline*}
 It remains to notice that since   $\lim_{n\to\infty}\sqrt[n]{\EE(Z^{n})}= z$,  for all large enough $n$ we have $\EE(Z^{n})>(z-\eps/2)^n$. So
 $$
\frac{\EE(Z^{n}I_{Z\leq z-\eps})}{\EE(Z^{n})}\leq \frac{(z-\eps)^n}{(z-\eps/2)^n}\to0.
 $$
 Thus $\liminf_{n\to\infty}\frac{\EE(Z^{n+k})}{\EE(Z^n)}\geq (z-\eps)^k$, and since $\eps>0$ was arbitrary,
 this ends the proof of Lemma \ref{Lem-int}.
\end{proof}

We now return to Step 2 of the proof of Theorem \ref{Thm-semin-u}. We apply Lemma \ref{Lem-int} with $Z=1+\hat u +\sqrt{1-q}\hat Z_{\hat u}$ and $n=N-K\to\infty$. We  note that,
 as in the proof of Theorem \ref{T-LDP}, from \eqref{ell(t)} we  get
 $1+\hat u +\sqrt{1-q}\hat Z_{\hat u}\geq0$.
 Next, we  compute $\zeta=\|1+\hat u +\sqrt{1-q}\hat Z_{\hat u}\|_\infty=1+\hat u+ \sqrt{1-q}\,\hat z$,
where $\hat z$ is the maximum of the support of  $\hat Z_{\hat u}$.
The calculation  is based on \eqref{u(t)}, and we get
\begin{equation}
  \label{hat-z}
  \sqrt{1-q}\hat z=U(\hat u)=\begin{cases}
    \hat C+\hat u/\hat C &\mbox{ if $\hat u<\hat C^2$},\\
    2 \sqrt{\hat u} & \mbox{ if $\hat C^2\leq \hat u\leq 1/\hat A^2$},\\
    \hat A\hat u+1/\hat A &\mbox{ if $\hat u>1/\hat A^2$}.
  \end{cases}
\end{equation}

Using \eqref{Eq:def-p_K} and Lemma \ref{Lem-int}    we conclude that there exists a unique polynomial $p_K$ of degree $K$
with coefficients determined solely by $s_1,\dots,s_K$ and parameters $\hat A,\hat B,\hat C,\hat D$ such that
\begin{equation}
  \label{Lim-hatZ-computed} \lim_{N\to\infty}\frac{\EE\left( (1+\hat u+\sqrt{1-q}\hat Z_{\hat u})^{N-K}\prod_{j=1}^K(1+s_j+\sqrt{1-q}\hat Z_{s_j})\right)}
  {\EE\left((1+\hat u+\sqrt{1-q}\hat Z_{\hat u})^N\right)}
   = \frac{p_K(\zeta)}{\zeta^K},
\end{equation}
where
\begin{equation}\label{zeta}
 \zeta=1+\hat u+\sqrt{1-q}\hat z =\begin{cases}

\frac{(C+1) (C+u)}{C u} & \mbox{if $u<C^2$ }, \\
  (1+1/\sqrt{u})^2 & \mbox{if $C^2\leq u\leq 1/A^2$ }, \\

 \frac{(A+1) (A u+1)}{A u} & \mbox{if $u>1/A^2$ } .\end{cases}
\end{equation}

\comment{added "transition"}
Now that we determined the limit on the right hand side of \eqref{Lim-hatZ+}, our next step is to represent it in terms of a Markov process.
\subsubsection*{Step 3: Process representation of the limit \eqref{Lim-hatZ+}} We now introduce the process that
 appears on the right hand side of \eqref{Lim-hatZ+}.
Let $(\hat{\tilde Z}_t)_{t\geq \hat u}$ be a Markov process with the same transition probabilities as process $(\hat Z_t)$,
but started at the deterministic point
$\hat{\tilde Z}_{\hat u}=\hat z$.
That is,
\begin{itemize}
  \item[(i)] $(\hat{\tilde Z}_t)_{t\geq \hat u}$ arises from an auxiliary process $(\hat{\tilde Y}_t)_{t\geq \hat u}$ via
  (``hatted and tilded") formula \eqref{JW-Z};
  \item [(ii)] $(\hat{\tilde Y}_t)_{t\geq \hat u}$ is a Markov process started at the deterministic point
  $$\hat{\tilde
      Y}_{\hat u}=\frac{\sqrt{1-q}\hat z}{2\sqrt{\hat u}};$$
  \item[(iii)] for $\hat u\leq s<t$ the transition probabilities of $(\hat{\tilde Y}_t)_{t\geq \hat u}$
  are given by formula \eqref{Y-transitions} with $A,B$ replaced by $\hat A$ and $\hat B$;
\item[(iv)] the univariate distributions of $(\hat{\tilde Y}_t)_{t\geq \hat u}$ are given by
formula \eqref{Y-transitions} with $s=\hat u$,  $x=\sqrt{1-q}\hat z/(2\sqrt{\hat u})$ and with $A,B$ replaced by $\hat A$ and $\hat B$.
\end{itemize}
 From this description, it is clear that $(\hat{\tilde Y}_t)_{t\geq \hat u}$ is an Askey-Wilson process with parameters
$\hat{\tilde{A}}, \hat{\tilde{B}}, \hat{\tilde{C}},\hat{\tilde{D}}$, which can be expressed in terms of the original parameters $A,B,C,D$ as follows:

\begin{eqnarray}
\hat{\tilde{A}}&=&\hat A=C, \nonumber
\\
\hat{\tilde{B}}&=&\hat B=D,  \nonumber
\\
 \hat{\tilde{C}}&=&
\sqrt{\hat u} (x+\sqrt{x^2-1})=\begin{cases}
C/u & \mbox{if } C^2>u ,\\

1/\sqrt{u} & \mbox{if } C^2\leq u\leq 1/A^2, \\
A & \mbox{if } u>1/A^2,
  \end{cases} \label{(5.11)}
\\
\hat{\tilde{D}}%
&=&\sqrt{\hat u} (x-\sqrt{x^2-1})=\hat u/\hat{\tilde{C}}.   \nonumber
\end{eqnarray}
(The calculations here are based on \eqref{hat-z}.)
In particular, since $\hat{\tilde{C}}\hat{\tilde{D}}=\hat u$ and  $\hat{\tilde{A}}\hat{\tilde{B}}\leq 0$, from \eqref{EQ:I}
we read out that  the process $(\hat{\tilde Y}_t)_{t\in I}$ is indeed  defined on the
time-interval $I=[\hat u,\infty)$, as expected.

Since process $(\hat{\widetilde Z}_t)_{t\geq \hat u}$ has the same transition probabilities  as process $(\hat Z_t)$, both processes have the same
  martingale polynomials. Therefore,
by the same algebra as before,    formula \eqref{Eq:def-p_K} takes the following form

\begin{equation}
  \label{Eq:def-p_K+}
  \EE( \prod_{j=1}^K(1+s_j+\sqrt{1-q}\hat{\tilde{Z}}_{s_j})|\hat Z_{\hat u})=p_K(1+\hat u+\sqrt{1-q}\hat{\tilde{Z}}_{\hat u}).
\end{equation}
However, $\hat{\tilde{Z}}_{\hat u}=\hat z$ is non-random, so by integrating \eqref{Eq:def-p_K+} with respect to the degenerate law,  the right hand side of
\eqref{Lim-hatZ-computed} can be written as
$$\EE\left( \prod_{j=1}^K(1+s_j+\sqrt{1-q}\hat{\tilde{Z}}_{s_j})\right)/\zeta^K$$
and \eqref{Lim-hatZ+} follows, as $\zeta^K=(1+\hat u+\sqrt{1-q}\hat z)^K=\EE(1+\hat u+\sqrt{1-q}\hat{\widetilde Z}_{\hat u})^K$.

\subsubsection*{Step 4: time inversion again}
To conclude the proof, we rewrite the right hand side of \eqref{Lim-hatZ} by the same time inversion that we already used in Step 1 of the proof. For $t=1/s\leq 1$ define
$$\widetilde Z_t=t\hat{\tilde{Z}}_{1/t}.$$
As previously, $\widetilde Z_u= \frac{1}{\hat u}\hat{\tilde{Z}}_{\hat u}$, compare \eqref{Z-inv}.  Using the correspondence \eqref{t2s} again, the numerator of the fraction on the right hand side of \eqref{Lim-hatZ} becomes
\begin{multline*}
 \EE\left(\prod_{j=1}^K(1+s_j+\sqrt{1-q}\,\hat{\tilde{Z}}_{s_j})\right) =
 \EE\left(\prod_{j=1}^K(1+1/t_j+\sqrt{1-q} \hat{\tilde{Z}}_{1/t_j})\right)
 \\
 =
\frac{1}{t_1\dots t_K} \EE\left(\prod_{j=1}^K(1+ t_j+\sqrt{1-q} t_j \hat{\tilde{Z}}_{1/t_j})\right)
\\
=\frac{1}{t_1\dots t_K} \EE\left(\prod_{j=1}^K(1+ t_j+\sqrt{1-q} \widetilde{Z}_{t_j})\right).
\end{multline*}
We use the last expression to replace the right hand side of \eqref{Lim-hatZ} and  we use \eqref{run-out-of-names} to replace the left  hand side of \eqref{Lim-hatZ}.
The products $t_1\dots t_K$ cancel out, and we get \eqref{LimZ}.

By the previous discussion about the time inversion, process
  $(\widetilde {Z}_t)_{0\leq t\leq u}$ is  the Markov process from \eqref{JW-Z} in Section \ref{AMP},
with parameters
$\widetilde{A}, \widetilde{B}, \widetilde{C},\widetilde{D}$ obtained by swapping the two pairs of the parameters for the process $(\hat{\tilde{Z}})$,
$$\widetilde{A}=\hat{\tilde{C}}, \widetilde{B}=\hat{\tilde{D}}, \widetilde{C}=\hat{\tilde{A}},\widetilde{D}=\hat{\tilde{B}}.$$
In particular, from \eqref{(5.11)} we get  \eqref{tilde-A-u} and \eqref{tilde-B-u}.

This concludes the proof of Theorem \ref{Thm-semin-u}.
\subsection{Limits as $K\to\infty$ or as $u\to\infty$}

Recall measures $\mu_{K,u}$ from the conclusion of Theorem \ref{Thm-semin-u}.
A natural plan to prove Large Deviations for these measures would be to compute the limit
\begin{equation}   \label{La-u}
\Lambda(\la):=\lim_{K\to\infty}\frac1{K}\log \int  \exp(\la\sum_{j=1}^K\tau_j)d\mu_{K,u}
\end{equation}
for all real $\la$. Unfortunately, large $\la$ are beyond the scope of our methods, so we only can state the result for small $\la$.
\begin{proposition}\label{P-LDP-u}
Suppose that the parameters of ASEP are such that $AC<1$   with
$A,B,C,D$ defined by \eqref{ABCD}, and fix $u\geq 1$. Then for $-\infty <\la\leq \log u$ the  limit \eqref{La-u} exists,
 $\Lambda(\la)$ does not depend on $u$ and is given by $\Lambda(\la)=\LL(\la)-\LL(0)$, where $\LL(\la)$ is defined in \eqref{L0-ans}.
\end{proposition}

\begin{proof}%
  The proof is based on \eqref{u-Z}.
As in  the proof of Theorem \ref{T-LDP}, we first confirm that $1+t+\sqrt{1-q}\widetilde Z_t \geq 0$.
(With $\widetilde B_u>0$ by \eqref{tilde-B-u}, here
 we have fewer atoms contributing to $\ell(u)$.)

Next, we compute $\|1+t+\sqrt{1-q}\widetilde Z_t\|_\infty$.
Since $\widetilde B_u>0$, we need to consider three possibilities for the atoms:

\begin{enumerate}
  \item \label{I} Atom at $I_1(t):=\left(\sqrt{t}\widetilde B_u+1/(\sqrt{t}\widetilde B_u)\right)/2$ for $t>1/\widetilde B_u^2$;
  \item  \label{II} Atom at $I_2(t):=\left(\sqrt{t}\widetilde A_u+1/(\sqrt{t}\widetilde A_u)\right)/2$ for $t>1/\widetilde A_u^2$;
\item \label{III} Atom at $I_3(t):=\left(\widetilde C/\sqrt{t}+\sqrt{t}/\widetilde C\right)/2$ for $t<\widetilde C^2_u$.
\end{enumerate}
From  \eqref{tilde-B-u} we read out that since $t\leq u$, case \eqref{I} never occurs. Indeed,
$$
t>1/\widetilde B_u^2=\begin{cases}
   C^2 &\mbox{only when $u<C^2$, so $t>u$, contradicting $t\leq u$},\\
   u & \mbox{contradicting $t\leq u$},\\
   A^2u^2 &\mbox{when $u>\tfrac1{A^2}$, but $u>t>A^2u^2$  implies that $u<\tfrac1{A^2}$}.
\end{cases}
$$

Atom \eqref{II} can contribute only when $u<C^2$ or when $u>1/A^2$.
However, only the second case contributes to the maximum.
This is because when $u<C^2$ both atoms \eqref{II} and \eqref{III} occur for $u^2/C^2<t<u$, but atom \eqref{III} is  then larger.
To verify the latter, we compute
$$
\frac{d}{dt}(I_2(t)-I_3(t))=\frac{(C^2-u)(t+u)}{2 C u t \sqrt{t}}>0.
$$
Since $I_2(u)=I_3(u)$ and $t\leq u$,  this shows that $I_2(t)\leq I_3(t)$, i.e.,   \eqref{III} is the larger atom.

To summarize, we have the following two cases. Either $t<C^2$ and then the atom is at $(C/\sqrt{t}+\sqrt{t}/C)/2$, or $t>1/A^2$ and
then necessarily $u>1/A^2$ with the atom
at $(A\sqrt{t}+1/(A\sqrt{t}))/2$.  Thus for $t=e^\la\leq u$ and $u\geq 1$ we see that the limit \eqref{La-u} exists
and  $\Lambda(\la)=\LL(\la)-\LL(0)$ is given by \eqref{L0-ans} as claimed.
\arxiv{ Indeed, if $t<C^2$ and  $u<C^2$ then atom from case \eqref{III} dominates and if $u>C^2$ then  only atom \eqref{III} contributes.
So $\|1+t+\sqrt{1-q}\widetilde Z_t\|_\infty=(C+t)(1+C)/C$, see \eqref{u(t)}.

On the other hand, if $t>1/A^2$ then $u>t>1/A^2$ so atom \eqref{II} contributes and $\|1+t+\sqrt{1-q}\widetilde Z_t\|_\infty=(1+At)(1+A)/A$, see \eqref{tilde-A-u} and \eqref{u(t)}.

There is no atom when $C^2\leq t \leq 1/A^2$, so $\|1+t+\sqrt{1-q}\widetilde Z_t\|_\infty=(1+\sqrt{t})^2$.
}

\end{proof}
\begin{remark}
 \citet{duhart2014semi}  use a
matrix representation \cite[Theorem
3.1]{grosskinsky2004phase} to prove the large deviations principle for the stationary measure of an ASEP with $q=0$.
In particular, in \cite[Corollary 6.1]{duhart2014semi} they show that  when $u=1$ and $q=0$ the
limit  \eqref{La-u} exists for all  $\la\in\RR$. When $\la>0$ and $C>1$,  their answer   is still given by
\eqref{L0-ans}, but with $A$ replaced by $\widetilde A_1$ as defined in \eqref{tilde-A-u}.
This suggests that perhaps the limit  \eqref{La-u} exists for all $\la\in\RR$ but depends on $u$ when $\la>\log u$.
\end{remark}
Our final result deals with $u\to\infty$.
\begin{proposition}\label{P-u-infty}
   As $u\to\infty$, measures $\mu_{K,u}$ converge weakly to a measure $\mu_{K,\infty}$ such that
     \begin{equation}\label{infty}
\int \prod_{j=1}^K t_j^{\tau j}d\mu_{K,\infty}= \frac{\EE\left(\prod_{j=1}^K(1+t_j+\sqrt{1-q}\,\widetilde{Z}_{t_j})\right)}{\EE\left((2+\sqrt{1-q}\,\widetilde{Z}_1)^K\right)},
\end{equation}
where  $(\widetilde {Z}_t)_{0\leq t<\infty}$ is  the Markov process from \eqref{JW-Z} in Section \ref{AMP} constructed using
parameters $\widetilde A_\infty=A$, $\widetilde B_\infty=0$, $\widetilde C_\infty=C$, $\widetilde D_\infty=D$.
\end{proposition}
(Recall that $CD\leq 0$ by \eqref{kappa}.)
\begin{proof}
Let  $(\widetilde Z_t^{(u)})_{0\leq t\leq u}$ denote the Markov process from Theorem \ref{Thm-semin-u}, with its dependence on $u$ written explicitly.
Since the joint moments of $\widetilde Z_t^{(u)}$ depend continuously on  the parameters \eqref{tilde-A-u} and \eqref{tilde-B-u} which converge as $u\to\infty$,
    it is clear that the probability generating functions %
   \eqref{u-Z}
    converge as $u\to\infty$. Thus
  $\mu_{K,u}\to \mu_{K,\infty}$ as $u\to\infty$ and \eqref{infty}  holds. The parameters of $(\widetilde {Z}_t)_{0\leq t<\infty}$ are calculated as limits of \eqref{tilde-A-u} and \eqref{tilde-B-u}.
\end{proof}
\begin{remark}
Proposition \ref{P-u-infty} shows that $\mu_{K,\infty}$ is a stationary measure of
 an ASEP on $\{1,\dots, K\}$ with parameters $\widetilde \alpha=\alpha$,
$$
 \widetilde \beta=\frac{2 \beta  (1-q)}{1-q+\beta +\delta +\sqrt{4 \beta  \delta +(\beta -\delta +q-1)^2}},
$$
$\widetilde \delta=0$, $\widetilde\gamma=\gamma$ and the same $q$.

In particular, if $\delta=0$ then
$$\widetilde{\beta}=\begin{cases}
  1-q & \mbox{if $q+\beta>1$}, \\
  \beta & \mbox{if $q+\beta\leq 1$},
\end{cases}
$$
so when $\delta=0$ and $\beta\leq 1-q$ after taking the limits $N\to\infty$ and then $u\to\infty$, we recover the initial ASEP.

 \end{remark}
 \section*{Acknowledgement} %
JW research was supported in part by  grant 2016/21/B/ST1/00005 of the National Science Center, Poland.
  WB research was supported in part by the Taft Research Center at the University of Cincinnati. We thank Amir Dembo for  a helpful discussion and references on the semi-infinite ASEP.

\appendix
\newcommand{\yxts}{(y;x,t,s)}
\newcommand{\qnum}[1]{\left[#1\right]_{q}}                                      %
\newcommand{\qfact}[1]{\left[#1\right]_{q}!}                                    %
\newcommand{\qbin}[2]{\left[\begin{array}{c}#1 \\ #2 \end{array}\right]_{q}}   %

\section{Proof of Theorem \ref{Thm-gen-biPoisson}}\label{S:PofThm}
Denote $\qnum{n}=1+q+\dots+q^{n-1}$.
Consider the family of monic polynomials $\{Q_n\yxts: n=0,1,\dots\}$ in variable $y$,
defined by the three step recurrence:
\begin{equation}
\label{Q-def}
y\ Q_n\yxts =Q_{n+1}\yxts+\calA_n(x,t,s)Q_n\yxts+\calB_n(x,t,s)Q_{n-1}\yxts,
\end{equation}
with initial polynomials $Q_{-1}\equiv0$, $Q_0\equiv1$, where the coefficients in the recurrence are
\begin{equation*}
  \calA_n(x,t,s)=q^n x+\qnum{n}\left(t\eta+\theta-\qnum{2}q^{n-1}s\eta\right)\label{AAA},
\end{equation*}
\begin{equation*}
  \calB_n(x,t,s)=\qnum{n}\left(t-sq^{n-1}\right)\left\{1+\eta x
q^{n-1}+\qnum{n-1}\eta\left(\theta-s\eta q^{n-1}\right)\right\}\label{BBB},
\end{equation*}
for $n\ge0$. (For $n=0$, the formulas should be interpreted as $\calA_0(x,t,s)=x$, $\calB_0(x,t,s)=0$.)

It is known, see \cite{Bryc-Matysiak-Wesolowski-04b}, that family $\{Q_n\yxts: n=0,1,\dots\}$ is orthogonal with respect to the transition probabilities
$\{P_{s,t}(x,dy)\}$ of the bi-$q$-Poisson process, i.e.
\begin{equation}
  \label{Q-orth}
  \int Q_n\yxts Q_m\yxts P_{s,t}(x,dy) =0 \mbox{ for $m\ne n$}.
\end{equation}
Polynomials $M_n(y;t):=Q_n(y;0,t,0)$ are martingale polynomials, i.e.,
\begin{equation}
  \label{mart}
  \int M_n(y;t)P_{s,t}(x,dy)=M_n(x;s)
\end{equation} and \eqref{Q-def} simplifies to
 the three step recurrence:
\begin{equation}\label{Q-Poisson}
x M_n(x;t) \\=M_{n+1}(x;t)+(\theta+t\eta)[n]_qM_n(x;t)+t(1+\eta\theta[n-1]_q)[n]_qM_{n-1}(x;t),
\end{equation}
$n\geq 0$, with $M_{-1}=0,M_0=1$. In  particular, it is clear that $M_n(x;t)$ is a polynomial in both $x$ and $t$.

In view of \eqref{mart}, the action of the infinitesimal generator on $M_n(y;t)$ is simply
\begin{equation*}
  A_t(M_n(\cdot;t))(x)=-\frac{\partial}{\partial t}M_n(x;t);
\end{equation*}
this is easiest seen by looking at the left-generator, compare \cite[Lemma 2.1]{Bryc-Wesolowski-2013-gener}. Note that by linearity this determines action of $A_t$ on all polynomials:
for $p(x)=\sum_{k} a_k(t)M_k(x;t)$, we have
\begin{equation}
  \label{A(expand)}
  A_t(p)(x)=-\sum_k a_k(t) \frac{\partial}{\partial t}M_k(x;t).
\end{equation}

It is enough to determine action on polynomials of the related operator $H_t(p)(x):=A_t(x p(x))-x A_t(p)(x)$, compare \cite[Eqtn. (13)]{Bryc-Wesolowski-2013-gener}, and
it is enough to determine action of $H_t$ on polynomials $M_n(x;t)$.
\begin{lemma}
  \label{L:HM}
 \begin{equation}
   \label{HonM}
    H_t(M_n(\cdot;t))(x)=\eta \qnum{n} M_n(x;t)+(1+\eta\theta\qnum{n-1})\qnum{n}M_{n-1}(x;t).
 \end{equation}
\end{lemma}
\begin{proof}
  This proof is a minor variation of the proof of \cite[Lemma 2.2]{Bryc-Wesolowski-2013-gener}: we use \eqref{Q-Poisson} to write
  $x M_n(x;t)$ as a linear combination of $\{M_k(x;t)\}$, apply \eqref{A(expand)}, and then subtract
  $$x A_t(M_n(\cdot;t))(x)=-x\frac{\partial}{\partial t}M_n(x;t)=-\frac{\partial}{\partial t}\left(xM_n(x;t)\right),$$
  where we again use recursion \eqref{Q-Poisson}.
\end{proof}
\begin{lemma}
  \label{L:H-rep} If $p$ is a polynomial, then
  \begin{equation}
    \label{H-integral}
    H_t(p)(x)=(1+\eta x) \int \frac{p(y)-p(x)}{y-x} P_{q^2 t, t}(q(x-t\eta)+\theta,dy).
  \end{equation}
\end{lemma}
\begin{proof}
 Fix $t>0$.
 The first step is to go back to recursion \eqref{Q-def}, and notice that $\nu_{x,t}(dy):=P_{q^2 t, t}(q(x-t\eta)+\theta,dy)$
 is the orthogonality measure of the polynomials $W_n(y;x,t)=Q_{n+1}(y;x,t,t)/(y-x)$, $n=0,1,\dots$. This is because polynomials $\{W_n(y;x,t)\}$
 satisfy the three step recursion
 \begin{multline}
  y W_n(y;x,t)=W_{n+1}(y;x,t)+ \left(q^{n+1}x+\qnum{n+1}(t\eta+\theta-\qnum{2}q^n\eta t)\right)W_n(y;x,t)
  \\ + t\qnum{n+1}(1-q^n)\left(1+\eta x q^n +\qnum{n}\eta(\theta-t\eta q^n)\right)W_{n-1}(y;x,t),
 \end{multline}
which is derived from \eqref{Q-def} with $s=t$.

Clearly, \eqref{H-integral} holds for a constant polynomial.  By linearity, it is enough to verify \eqref{H-integral} for
 $p(x)=M_n(x;t)$ with $n\geq 1$, in which case the left hand side is given by \eqref{HonM}.
We want to show that the right hand side is given by the same expression.

For a fixed  $s<t$ and $n\geq 1$, we write
 polynomial $M_n(y;t)$ as a linear combination of the (monic) polynomials $\{Q_k\yxts:k\geq 0\}$ with coefficients $\{b_{n,k}(x,t,s):k=0,\dots,n\}$.
 Since $Q_0\yxts=1$ and  $Q_1\yxts=y-x$,  and $\nu_{x,t}(dy)$ is a probability measure, we get
\begin{multline*}
  \int \frac{M_n(y;t)-M_n(x;t)}{y-x}\nu_{x,t}(dy)=\sum_{k=1}^n b_{n,k}(x,t,s)\int \tfrac{Q_k\yxts-Q_k(x;x,t,s)}{y-x}\nu_{x,t}(dy)
  \\=b_{n,1}(x,t,s)+\sum_{k=2}^n b_{n,k}(x,t,s)  \int \frac{Q_k\yxts-Q_k(x;x,t,s)}{y-x}\nu_{x,t}(dy).
\end{multline*}
Measure $\nu_{x,t}(dy)$ has compact support so we can pass to the limit as $s\to t$ under the integral. It is also known
 that  $ b_{n,k}(x,t,s)$ are continuous functions of $s$; in fact,
 from the explicit formulas for $\widetilde{b^{(n)}_{n-k}}(y;x,t,s)$ in \cite[page 627]{Bryc-Matysiak-Wesolowski-04b} we see that
 $ b_{n,k}(x,t,s)=\widetilde{b^{(n)}_{n-k}}(0;x,0,s)$ does not depend on $t$ and is a polynomial in $s$.
Noting that
$Q_n(x;x,t,t)=0$ for $n\geq 1$, see recursion \eqref{Q-def}, and recalling the definition of polynomials $\{W_k\}$,
we see that as $s\to t$ the   sum  converges to
$$\sum_{k=2}^n b_{n,k}(x,t,t) \int W_{k-1}(y;x,t)\nu_{x,t}(dy)$$ which is $0$ by  orthogonality.
Therefore,
\begin{equation}
  \label{*}(1+\eta x) \int \frac{M_n(y;t)-M_n(x;t)}{y-x}\nu_{x,t}(dy)=\lim_{s\to t}(1+\eta x) b_{n,1}(x,t,s).
\end{equation}

We now analyze the right hand side of \eqref{*}.  By orthogonality \eqref{Q-orth} we see that
$$b_1(x,t,s)=\frac{\int Q_1\yxts M_n(y;t)P_{s,t}(x,dy)}{\int Q_1^2\yxts P_{s,t}(x,dy)}.$$
Since $Q_1\yxts=y-x$, and   \cite[(2.27)]{Bryc-Matysiak-Wesolowski-04} implies that $\int (y-x)^2P_{s,t}(x,dy)=(t-s)(1+\eta x)$,
 we get
\begin{multline*}
(1+\eta x)b_1(x,t,s)=\frac{1}{t-s}\int (y-x) M_n(y;t)P_{s,t}(x,dy) \\=
\frac{1}{t-s}\left(\int y M_n(y;t)P_{s,t}(x,dy)-x \int   M_n(y;t)P_{s,t}(x,dy)\right).
\end{multline*} Using the three step recursion \eqref{Q-Poisson} and the martingale property of  polynomials $\{M_k(y;t)\}$ we get
\begin{multline*}
(1+\eta x)b_1(x,t,s)\\=
\frac{1}{t-s}\Big(\left(M_{n+1}(x;s)+(\theta+t\eta)[n]_qM_n(x;s)+t(1+\eta\theta[n-1]_q)[n]_qM_{n-1}(x;s)\right) \\
-
\left(M_{n+1}(x;s)+(\theta+s\eta)[n]_qM_n(x;s)+s(1+\eta\theta[n-1]_q)[n]_qM_{n-1}(x;s)\right)\Big)\\
=
 \eta [n]_qM_n(x;s)+(1+\eta\theta[n-1]_q)[n]_qM_{n-1}(x;s).
\end{multline*} In particular, $(1+\eta x)b_1(x,t,s)$ does not depend on $t$, and the right hand side of \eqref{*}
  matches the right hand side of  \eqref{HonM}. By linearity, this proves \eqref{H-integral} for all polynomials $p$.

\end{proof}

\begin{proof}[Proof of Theorem \ref{Thm-gen-biPoisson}]
From Lemma \ref{L:H-rep} we see that $H_t$ is given by \eqref{H-integral}.
 From \cite[Lemma 2.4]{Bryc-Wesolowski-2013-gener} multiplied by $1+\eta x$,
 we see that $A_t$ is then given by  \eqref{A}.
\end{proof}


\end{document}